\documentclass[10pt,reqno,a4paper]{amsart}

\usepackage{mathtools}
\usepackage{amsmath,amsthm, amsfonts, amsrefs, amssymb}
\mathtoolsset{showonlyrefs}

\usepackage{graphicx}
\usepackage{bbm}
\usepackage{tikz}
\usepackage{tikz-cd}
\usepackage{hyperref}
\usepackage{todonotes}
\usepackage{mathrsfs}
\usepackage[margin=2cm]{geometry}
\usepackage{enumitem}

\numberwithin{equation}{section}

\theoremstyle{plain}
\newtheorem{thm}{Theorem}[section]
\newtheorem{lemma}[thm]{Lemma}
\newtheorem{col}[thm]{Corollary}

\theoremstyle{definition}
\newtheorem{defn}[thm]{Definition}

\newtheorem{remark}[thm]{Remark}

\def\XXint#1#2#3{{\setbox0=\hbox{$#1{#2#3}{\int}$ }
\vcenter{\hbox{$#2#3$ }}\kern-.6\wd0}}

\newcommand{\nc}{\newcommand}

\nc{\sgn}{\mathrm{sgn}}
\nc{\nothing}{\varnothing}
\nc{\Ab}{\mathrm{Ab}}
\nc{\Graph}{\mathbf{Graph}}
\nc{\FinSet}{\mathbf{FinSet}}
\nc{\Vect}{\mathbf{Vect}}
\nc{\Top}{\mathbf{Top}}
\nc{\Cat}{\mathbf{CAT}}
\nc{\Set}{\mathbf{Set}}
\nc{\Rel}{\mathbf{Rel}}
\nc{\Grp}{\mathbf{Grp}}
\nc{\AbGrp}{\mathbf{AbGrp}}
\nc{\cT}{\mathcal T}
\nc{\cG}{\mathcal G}
\nc{\cF}{\mathcal F}
\nc{\cK}{\mathcal{K}}
\nc{\cE}{\mathcal E}
\nc{\cP}{\mathcal P}
\nc{\cM}{\mathcal M}
\nc{\cC}{\mathcal C}
\nc{\cB}{\mathcal B}
\nc{\cS}{\mathcal S}
\nc{\cD}{\mathcal D}
\nc{\cA}{\mathcal A}
\nc{\PP}{\mathbb P}
\nc{\cU}{\mathcal U}
\nc{\Mod}{\operatorname{Mod}}
\nc{\Aut}{\operatorname{Aut}}
\nc{\del}{\partial}
\nc{\inter}{\mathrm{o}}
\nc{\close}[1]{\overline{#1}}
\nc{\pderiv}[2]{\frac{\partial #1}{\partial #2}}
\nc{\tr}{\operatorname{tr}}
\nc{\dd}{\mathrm{d}}
\renewcommand{\epsilon}{\varepsilon}

\nc{\RP}{\R\mathrm{P}}

\newcommand{\R}{\mathbb{R}}

\newcommand{\N}{\mathbb{N}}

\newcommand{\inv}{^{-1}}

\nc{\ev}{\mathrm{ev}}
\nc{\Nat}{\mathrm{Nat}}

\newcommand{\ra}{\rightarrow}

\newcommand{\diver}{\mathrm{div}}
\nc{\vect}[1]{\mathbf{#1}} 
\nc{\im}{\mathrm{im}}
\nc{\weakra}{\rightharpoonup}
\nc{\Cof}{\mathrm{Cof}}
\nc{\innprod}[2]{\left\langle #1, #2 \right\rangle}
\nc{\norm}[1]{\left\| #1 \right\|}
\nc{\abs}[1]{\left\lvert #1 \right\rvert}
\nc{\e}[1]{ \mathbb E \left[ #1 \right] }
\renewcommand{\d}{\mathrm{d}}
\newcommand{\loc}{\mathrm{loc}}
\nc{\vp}{\varphi}

\newcommand{\lltriangle}{\rotatebox[origin=c]{360}{\tikz{\draw[thick](0,0)--(0.25,0)--(0,0.25)--cycle;}}}

\allowdisplaybreaks

\subjclass[2020]{35R35, 35K65}
\keywords{Porous media equation, thin-film equation, qualitative properties, support propagation, waiting time phenomenon}

\begin{document}

\title[Free Boundary Propagation of a Thin-Film Equation with Nonlinear Diffusion]{Free Boundary Propagation of a Thin-Film Equation with Nonlinear Diffusion: Scale-Consistent Methods for Optimal Results}
\author{Joshua Utley}
\address{Department of Mathematics, FAU Erlangen-Nürnberg, Cauerstraße 11, 91058 Erlangen}
\email{josh.utley@fau.de} 
\date{\today}

\begin{abstract}
We investigate the free boundary propagation of a non-linear degenerate parabolic problem with both fourth and second order non-linearities, which may be used to model viscous thin-film flow in the regime of weak slippage with an additional contributions from a singular disjoining pressure brought forth by inter-molecular forces. By carefully treating the individual scaling behaviors of the non-linearities, we obtain asymptotically optimal upper and lower bounds on the rate of forward support propagation. These rates show that the fourth-order operator dominates the propagation rate on short time scales, but the long-term rate of propagation is controlled by the second-order operator. We estimate further the time at which this trade-off in propagation rate occurs. Lastly, we examine sufficient conditions to ensure that the free boundary does not immediately advance in a given direction, leading to a detailed picture of how the competition between the second/fourth-order operators influence the free boundary dynamics of solutions.
\end{abstract}

\maketitle

\section{Introduction}

The thin-film equation
\begin{equation}\label{eq-tfe}\tag{TFE}
u_t + \diver(u^n\nabla\Delta u) = 0
\end{equation}
and the porous-medium equation 
\begin{equation}\label{eq-pme}\tag{PME}
u_t - \Delta(u^m) = 0
\end{equation}
are quintessential examples of degenerate parabolic PDEs which give rise to free boundary problems, and a large body of literature exists in which the motions of free boundaries for their solutions are examined. Concerning the porous-medium equation, the evolution of the support and its boundary is well-understood (s. \cite[Chapter 14]{Vazquez_PME} for a thorough overview). For the thin-film equation, the picture regarding free-boundary propagation has also become increasingly well-resolved (see, for instance, \cite{BernisFSOP,Shishkov/Hulshof,DPGG,GGDP,BasicFSOP,GrnHabWTP,FischerWTU,Fischer2016}). The object of this work is to examine in which ways the competition of the two operators influence the support propagation of solutions to the following PDE:
\begin{equation}\label{eq-TFPM}\tag{TFPM}
u_t + \diver(u^n\nabla\Delta u-u^{m-1}\nabla u) = 0.
\end{equation} 
Throughout, we consider solutions to the Cauchy problem associated with \eqref{eq-TFPM} in dimensions $d=1,2,3$ under the assumption that $(n,m)\in [2,3)\times (1,2)$. Concerning the free-boundary propagation of solutions, we address the following questions: 
\begin{enumerate}
\item\label{q1} Do weak solutions to \eqref{eq-TFPM} exhibit finite speed of forward propagation, i.e., does a positive time elapse before the support arrives at any point where it was not previously?
\item\label{q2} If the answer to the above question is yes, what can we then say about the speed of propagation? How does this speed depend on the interplay between the parameters $m,n$?
\item\label{q3} Given a point $x_0\in \Omega\setminus  \mathrm{supp}\, u_0$, we may define $R_0(x_0) := \mathrm{dist}\left( x_0, \mathrm{supp}\, u_0\right)$. What are sufficient assumptions which ensure that a positive time elapses before mass enters $B_{R_0}(x_0)$?
\end{enumerate}

The answers obtained in this work (s. Section \ref{section-main-results} for precise statements) shed light on the influence which the addition of the lower-order term has on the behavior of droplets spreading on a flat surface. Namely, the fourth-order part is dominant on short time-scales, but on asymptotic time-scales, the boundary of the support is confined to an annulus, the radii of which evolve in accordance with self-similar solutions to the porous-medium equation. 

\eqref{eq-TFPM} has been proposed as a model for the spreading of a liquid thin film under the influence of inter-molecular forces which diminish sufficiently quickly near the free boundary. We briefly summarize the exposition of \cite{BP}: For many liquid thin films, it makes sense to consider, in addition to the evolution which is driven by surface tension, the influence of long-range Van der Waals forces on the velocity of the fluid bulk. Following the approach of \cite{Williams_Davis_Nonlinear_Rupture}, the following evolution equation was obtained for the film-height (assuming a no-slip boundary condition):
\begin{equation}
    u_t + \diver\left( u  \vect U  \right) = 0, \quad \vect U = \frac{u^2}{3\eta}\left( \nabla \Pi (u) + \gamma\nabla\Delta u \right),
\end{equation}
where $\eta>0$ is the viscosity, $\gamma>0 $ the surface tension, and $\Pi(u)$ the so-called disjoining pressure, which is meant to encapsulate the influence of inter-molecular forces. Assuming that the disjoining pressure for Van der Waals forces decays according to the power law 
\[
     \Pi (u) \sim u^{-3},
\]
the following evolution equation appears (in dimensionless form):
\begin{equation}\label{eq-log-diffusion}
    u_t + \Delta(\log u) + \diver\left(u^{3}\nabla \Delta u\right) = 0.
\end{equation}

The logarithmic term is not what one wishes to see in the study of droplets which truly touch down, since it is no longer clear what is to be made of this equation. This term arises from the presumption that $\Pi(u) \sim u^{-3}$, which is generally accepted to be correct down to a certain length scale (s. \cite{Overview_Ref_Disjoining_Pressure}). However, as is noted by the authors in \cite{Williams_Davis_Nonlinear_Rupture}, the approximation becomes increasingly poor where the film-height is very small. In this case, the loss of regularity at the free boundary is so great that it is not even clear that non-negative weak solutions to \eqref{eq-log-diffusion} should exist at all. The ansatz of \cite{BP} is to replace the power law for $\Pi$ by a law of the following form:
\[
    \abs{\nabla \Pi(u)} \, \sim \, \varphi_m\left(u \right) \, \abs{\nabla\left(u^{-3}\right)},
\]
where $\varphi_m$ is a suitable superlinear cut-off function ($\varphi_m \sim u^m$), which activates below a certain film-height. This staves off the logarithmic singularity in \eqref{eq-log-diffusion} by enforcing that the Van der Waals forces are ``cut off'' below a certain length scale. The object of our study, the equation \eqref{eq-TFPM}, is a simplified version of the resulting evolution equation, which captures the new aspects of the evolution induced by the degeneracy of the disjoining pressure near the free boundary.

In addition to being an accessible free boundary problem---one might reasonably expect the existence of non-negative solutions with finite speed of propagation---solutions to \eqref{eq-TFPM} have also been shown to have some remarkable qualities, chief among them being the forward propagation of support for $n =3$ (s. \cite[Proposition 7]{BP}). This stands in stark contrast to expectations for solutions to \eqref{eq-tfe} alone when $n= 3$, where one encounters the famous no-slip paradox. This paradox, loosely put, says that a moving contact line entails an infinite dissipation of energy. We recommend the reader to the work \cite{Giacomelli_Durastanti_2023} and the references contained therein for a thorough overview of the no-slip paradox and some of the attempts (which include the introduction of \eqref{eq-TFPM}) to resolve it for thin-film equations. For $n=3$, where it is expected that solutions to \eqref{eq-tfe} exhibit no support propagation, we might reasonably expect that any free boundary motion is dominated by the porous-medium term. It is less clear what should happen when $n<3$. In this case the thin-film operator does \emph{not} fix the free boundary: it is known that sufficiently regular weak solutions to \eqref{eq-tfe} will also have forward support propagation (s. \cite{Fischer_Asymptotic}). Consequently, a competition between the porous-medium and thin-film operators arises, which it is the goal of this work to inspect. Some results concerning the outcome of this competition are available when $n< 2$ (s. \cite{GSDP}), and the results presented of this work provide a more complete picture when $n\geq 2$.

\subsection{Definitions and Statements of Main Results}\label{section-main-results}

This subsection is dedicated to the prerequisite definitions and rigorous formulations of the main results of this work. In the subsequent section, we will contextualize these results within the broader literature. We are interested in weak, energy-dissipating, zero-contact-angle solutions to the following PDE:
\begin{equation}\label{eq-TFPM-full}
\begin{cases}
u_t + \diver(u^n\nabla\Delta u- u^{m-1}\nabla u) = 0 & (t,x)\in \R_+\times \R^d \\
u(0,x) = u_0(x) & x\in \R^d
\end{cases}
\end{equation}
We make the following assumption on the initial data:
\begin{equation}\label{ass-u0}\tag{A1}
u_0 \in H^1(\R^d), \quad u_0 \geq0, \quad \mathrm{supp} \, u_0 \subset\subset  \R^d.
\end{equation}
For the mobility exponent $n$ and the degeneracy parameter $m$, we impose the following condition:
\begin{equation}\label{ass-nm}\tag{A2}
(n,m) \in [2,3)\times (1,2).
\end{equation}
Our notion of a solution is a perturbation of the solution concept given for thin-film equations with $n\in [2,3)$ (s. \cite{GrnHabC}). 

\begin{defn} \label{defn-solution}
 Let $d\in\left\{1,2,3\right\}$ and assume \eqref{ass-u0}. A non-negative function $u\in  L^\infty\left(\R_+; H^1(\R^d)\right)\cap L^\infty\left(\R_+; L^1(\R^d) \right)$ is called a \emph{strongly energy dissipating solution} to \eqref{eq-TFPM-full} with initial data $u_0\in H^1(\R^d)$ if the following properties are satisfied for all $p>\frac{4d}{2d-n(d-2)}$:
\begin{enumerate}
\item For any compact $K\subset\R^d$, $u\in H^1\left(\R_+; \left(W^{1,p}(K)\right)'\right)$. 
\item For all $T>0$ and $\varphi\in L^2\left(0,T; W^{1,\infty}(\R^d)\right)$ such that $\bigcup_{t\in[0,T]} \mathrm{supp}\, \varphi(t,\cdot)$ is compactly supported, we have \\
\begin{equation}\label{eq-weakform}
\int_0^T \innprod{u_t}{\varphi}_{(W^{1,p})'\times W^{1,p}} = \int_0^T\int_{\left\{u(t)>0\right\}} \left(u^n \nabla\Delta u - u^{m-1}\nabla u\right) \cdot \nabla\varphi .
\end{equation}
\item $u(0,x) = u_0(x)$ in the sense that $\lim_{t\ra 0}\norm{u(t,\cdot) - u_0}_{L^1_\loc} = 0$.
\item For all $\alpha\in (\max\left\{\frac{1}{2}-n,-m,-1\right\}, 2-n) \setminus \{0\}$, there exists $C_{\alpha,n,m}>0$ such that for all $\zeta\in C^\infty_c(\R^d)$, the following estimate holds for any $0\leq t_0\leq t_1$:
\begin{equation}\label{eq-alphaentropy}
\begin{split}
& \frac{1}{\alpha(\alpha+1)}\int_{\R^d} \zeta^4 u(t)^{\alpha+1}\bigg\vert_{t_0}^{t_1} \, + \, \int_{t_0}^{t_1}\int_{\R^d} \zeta^4 \abs{\nabla u^{\frac{\alpha+n+1}{4}} }^4 \\
& \hspace{2cm} +\, \int_{t_0}^{t_1}\int_{\R^d}  \zeta^4 \abs{D^2 u^{\frac{\alpha+n+1}{2}}}^2 \, + \, \int_{t_0}^{t_1}\int_{\R^d}  \zeta^4 \abs{\nabla u^{\frac{\alpha+m}{2}}}^2 \\ 
&\hspace{2cm}\lesssim_{\alpha,m,n} \,  \int_{t_0}^{t_1}\int_{\R^d} \left[ \abs{\nabla\zeta}^4 +  \zeta^2\abs{D^2\zeta}^2 \right]u^{\alpha+n+1} \, + \, \int_{t_0}^{t_1}\int_{\R^d} \abs{\Delta (\zeta^2)} u^{m+\alpha}.
\end{split}
\end{equation}
\item There exists $C_{n,m} > 0$ so that for all $\zeta\in C^\infty_c(\R^d)$, the following estimate holds for all $0\leq t_0\leq t_1$:
\begin{equation}\label{eq-energyest1}
\begin{split} 
& \int_\Omega \zeta^6\abs{\nabla u(t)}^2\bigg\vert_{t_0}^{t_1} +\int_{t_0}^{t_1}\int_{\R^d} \zeta^6 \abs{\nabla u^{\frac{n+2}{6}} }^6 +\int_{t_0}^{t_1}\int_{\R^d} \zeta^6 \abs{\nabla\Delta u^{\frac{n+2}{2}}}^2\\
& \qquad + \int_{t_0}^{t_1}\int_{\R^d} \zeta^6 \abs{D^2u^{\frac{m+1}{2}}}^2  + \int_{t_0}^{t_1}\int_{\R^d} \zeta^6 \abs{\nabla u^{\frac{m+1}{4}}}^4  \\
& \hspace{1.25cm} \lesssim_{n,m}  \int_{t_0}^{t_1}\int_{\R^d} \left[ \abs{\nabla\zeta}^6 + \zeta^3\abs{D^2\zeta}^3 \right]u^{n+2} + \int_{t_0}^{t_1}\int_{\R^d} \zeta^2\abs{\nabla\zeta}^4 u^{m+1}.
\end{split}\end{equation}
\end{enumerate}
\end{defn}

\begin{remark}
    Note that because $u\in L^\infty(\R_+; H^1(\R^d)) \cap H^1(\R_+; (W^{1,p}(K))') $ for every compact $K\subset \R^d$, we know that $u\in C(\R_+; L^1_{\loc}(\R^d))$ (see \cite{Simon}). 
\end{remark}

\begin{remark}\label{rem-existence}
On the subject of existence: Existence of solutions in one spatial dimension is obtained using the arguments presented in \cite{Cohen}, the only difference being that the estimate \eqref{eq-energyest1} is different from the localized $H^1$-estimate provided in \cite[Theorem 4]{Cohen}. It turns out that this difference comes down to happenstance. Examining their proof of the analogue to \eqref{eq-energyest1}, the terms which lead to attainment of $u^{\frac{m+1}{2}}\in H^2_\loc(\R^d)$ appear at the end of \cite[Eq. (3.32)]{Cohen}, but are not manipulated. A further integration by parts in this term produces exactly \eqref{eq-energyest1}. In higher spatial dimensions, the results presented in this work should be viewed as conditional, since no existence results are available. However, the techniques of \cite{GrnHabC} seem to be readily applicable, and so we suspect that a straightforward generalization of those techniques is sufficient to prove existence in dimensions two and three for initial data satisfying \eqref{ass-u0}.
\end{remark}

\noindent For the above class of solutions, we will investigate Questions \ref{q1}--\ref{q3}. Beginning with Question \ref{q1}, we must choose a notion of finite speed of propagation. We adopt the following definition:
\begin{defn}\label{defn-fsop}
Let $u\in C(\R_+; L^1_\loc(\R^d))$ be a non-negative function. Given an open ball $B_r(x_0)\subset \R^d$, \emph{the waiting time of $u$ associated with $B_r(x_0)$} is defined to be 
\begin{equation}\label{eq-waiting-time}
T_r(x_0) := \inf\left\{ t>0, \, \int_{B_r(x_0)} u(t,x) \ \d x >0 \right\}.
\end{equation} 
We say that \emph{$u$ has finite speed of propagation} if for every ball $B_r(x_0)$ satisfying $\overline{B_r(x_0)}\subset \R^d\setminus\mathrm{supp} \ u_0$, the corresponding waiting time is positive.
\end{defn}

The main object of study for the answers to Question \ref{q2} are the waiting times of individual points. We make sense of this idea by taking the limit of waiting times in open balls that decrease to the point in question.

\begin{defn}\label{defn-arrivaltime}
Let $u\in C(\R_+; L^1_\loc(\R^d))$ be a non-negative function. For a point $x_0\in \R^d\setminus \mathrm{supp} \, u_0$, the \emph{waiting time of $u$ at $x_0$} is defined to be 
\begin{equation}\label{eqn-arrivaltime}
T_w(x_0) := \lim_{r\ra 0} \, T_r(x_0).
\end{equation}
\end{defn}

We are now ready to present the main results of this work, starting with qualitative finite speed of propagation:

\begin{thm}\label{thm-fsop} Any strongly energy dissipating solution to \eqref{eq-TFPM-full} satisfying \eqref{ass-u0} and \eqref{ass-nm} has finite speed of propagation. 
\end{thm}

In relation to Question \ref{q2}, we provide rates of propagation in terms of upper and lower bounds on $T_w$. In order to answer Question \ref{q3}, recall that for a given point $x_0 \in \R^d\setminus \mathrm{supp}\, u_0$, we have defined 
\[
R_0(x_0) := \mathrm{dist}\left(x_0, \mathrm{supp}\, u_0 \right).
\]
The next result gives a lower bound for $T_w(x_0)$ along with sufficient conditions on the initial data such that mass does not immediately ``move towards $x_0$'', i.e., $T_{R_0}(x_0)>0$. In the following, we suppress the dependence of $T_w, T_{R_0}$, and $R_0$ on a given point $x_0$. 
\begin{thm}\label{thm-main-result-lower}   
Let $u$ be a strongly energy dissipating solution satisfying \eqref{ass-u0} and \eqref{ass-nm}. There exists a constant $c_1>0$ depending only on $(n,m,d)$ such that for any point $x_0\in \Omega \setminus \mathrm{supp} \ u_0$, the waiting time at $x_0$, $T_w$, admits the following lower bound: 
\begin{equation}\label{eq-spreadingrate-lower}
T_w\, \geq\, T_{R_0} + c_1\min\left\{ \norm{u_0}_{L^1}^{1-m} R_0^{d(m-1)+2}, \norm{u_0}_{L^1}^{-n} R_0^{dn+4} \right\}.
\end{equation} 
Furthermore, if there exist $\kappa = \kappa(x_0) \in (0,\infty) $ and $r_0 = r_0(x_0) \in (0,\infty)$ such that 
\begin{equation}
    \label{eq-wtp-flatness-condition}
    \sup_{r\leq r_0} \, \frac{r^2}{r^{d+\sigma}} \int_{B_{R_0+r}\setminus B_{R_0}} \abs{\nabla u_0}^2 \leq \kappa^2, \ \ \text{where } \  \sigma:= \max \left\{ \frac{8}{n}, \frac{4}{m-1} \right\}  ,
\end{equation}
then the time $T_{R_0}$ is also positive, and there exists $c_2>0$ depending only on $n,m,d$ such that 
\begin{equation}
    \label{eq-wtp-duration} T_{R_0} \, \geq \,c_2\min\left\{ r_0^{2+\frac{d(m-1)}{2}}\left( \kappa^2r_0^{d+\sigma} + \norm{u}_{L^\infty_t L^2_x}^2\right)^{\frac{1-m}{2}} \, , \, r_0^{4+\frac{dn}{2}}\left( \kappa^2r_0^{d+\sigma} + \norm{u}_{L^\infty_t L^2_x}^2\right)^{-\frac{n}{2}} \right\}.
\end{equation}

\end{thm}

\begin{remark}
    Note that because $n\in [2,3)$ and $m\in (1,2)$, we always have $\sigma = 4/(m-1)$ in the above result. The presentation above is merely to demonstrate that, in principle, contributions from both $n$ and $m$ should be considered. 
\end{remark}

In the case that $\mathrm{supp}\, u_0 \subset \subset \R^d$, Theorem \ref{thm-main-result-lower} allows us to construct, for given $x_0\in \R^d$, a function $R(x_0;t)$ such that $\mathrm{supp}\, u(t) \subset B(x_0,R(x_0;t))$ for all times $t>0$. This is the content of the following corollary\footnote{A weaker version of the following result and Theorem \ref{thm-fsop} were originally derived by the author in a master's thesis \cite{JU_MA}.}:

\begin{col}\label{col-biggest-ball}
    Let $u$ be a strongly energy dissipating solution satisfying \eqref{ass-u0} and \eqref{ass-nm}. Define the following time:
    \begin{equation}
        \label{eq-switching-time}
        t^* := c_1\norm{u_0}_{L^1}^{\frac{2(n+2-2m)}{d(n+1-m)+2}}.
    \end{equation}
    Then any $x_0\in \R^d$, the following holds for every $t>0$:
    \begin{equation}
        \label{eq-biggest-ball}
        \mathrm{supp}\, u(t) \subset \, B(x_0, R(x_0;t)),
    \end{equation}
    where the function $R(x_0;t)$ is defined by 
    \begin{equation}
        \label{eq-biggest-ball-radius}
        R(x_0; t) = R_0(x_0) + \mathrm{diam}(\mathrm{supp}\, u_0) + \begin{cases}
            \left(c_1\inv \norm{u_0}_{L^1}^n t \right)^{\frac{1}{dn+4}} & t\leq t^* \\
            \left(c_1\inv \norm{u_0}_{L^1}^{m-1} t \right)^{\frac{1}{d(m-1)+2}} & t> t^* 
        \end{cases}.
    \end{equation}
\end{col}

Theorem \ref{thm-main-result-lower} and Corollary \ref{col-biggest-ball} concern themselves with the maximum rate of propagation. In order to discern the optimality of the estimate \eqref{eq-spreadingrate-lower}, a lower bound on the propagation rate is necessary. This is the content of the following result, which verifies that the propagation rates derived here are asymptotically optimal. 

\begin{thm}
    \label{thm-main-result-upper}
     Let $d\in \{1,2,3\}$ and $u$ be a strongly energy dissipating solution satisfying \eqref{ass-u0} and \eqref{ass-nm}. Furthermore, assume that $m>4/3$ if $d=3$. Then there exists a constant $c_3>0$ depending only on $(n,m,d)$ such that for any $x_0\in \R^d\setminus \mathrm{supp}\, u_0$, we have: 
     \begin{equation}
         \label{eq-spreading-rate-upper}
         T_w \leq c_3\min\left\{ \norm{u_0}_{L^1}^{1-m} \left(R_0+\mathrm{diam}(\mathrm{supp}\, u_0)\right)^{d(m-1)+2}, \norm{u_0}_{L^1}^{-n}\left(R_0+\mathrm{diam}(\mathrm{supp}\, u_0)\right)^{dn+4} \right\} . 
     \end{equation}
        
\end{thm}

\subsection{Discussion of Results}

In order to properly contextualize the new results presented above, we begin this section by summarizing known results about solutions to \eqref{eq-pme}, \eqref{eq-tfe}, and \eqref{eq-TFPM}. Beginning with solutions of the porous-medium equation (with $m>1$), it is well-known that the extent of the free boundary is Hölder continuous in time with optimal exponent $\frac{1}{d(m-1)+2}$, which is established via comparison principles \cite{Aronson70,Knerr77}. It is also well-known (due to, e.g., \cite{Alikakos_PME_WTP,Chipot/Sideris}) that the optimal growth condition on the initial support that ensures that the support does not immediately move towards a given point (by this, we mean $T_{R_0}>0$) is given by 
\[
    \sup_{r\leq r_0} \, r^{-d-\frac{2}{m-1}} \, \int_{B_{R_0+r}\setminus B_{R_0}} u_0(x) \, \mathrm{d}x \,  < +\infty.
\]
Note that the condition is imposed on the mean value of $u_0$, which is not the case in Theorem \ref{thm-main-result-lower}, where a condition is imposed on the mean value of the squared gradient, which is generally stronger. Nonetheless, if one takes $u_0$ to be a power law near its boundary, i.e., $u_0 \sim \left(\abs{x-x_0 }-R_0(x_0)\right)_+^\gamma$ for some $\gamma>0$, then the above condition is equivalent to 
\[
      \sup_{r\leq r_0} \, r^{2-d-\frac{4}{m-1}} \, \int_{B_{R_0+r}\setminus B_{R_0}} \abs{\nabla u_0}^2 \, \mathrm{d}x \,  < +\infty .
\]

Note the the exponent of $r$ that appears above is precisely one of two exponents considered in the growth condition \eqref{eq-wtp-flatness-condition}.

Concerning solutions to \eqref{eq-tfe} for $n\in [2,3)$, similar results exist under certain regularity conditions, this usually being the so-called zero-contact-angle condition, which we we also enforce. For $n\in (0,3)$, this property may be deduced from the so-called $\alpha$-entropy estimate, which is identical to \eqref{eq-alphaentropy} without the terms depending on $m$ \cite{oldnonnegsol,DPGG}. For $n<2$, the $\alpha$-entropy estimate is also sufficient to prove finite speed of propagation, but it is unknown if this is possible for $n\geq 2$. All known results providing upper bounds on support propagation rates for $n\geq 2$ rely on energy estimates like \eqref{eq-energyest1}. Using energy-based techniques, it is known due to \cite{BernisFSOP,BernisFSOP2,Shishkov/Hulshof,BasicFSOP} that the extent of the free-boundary of weak solutions is Hölder continuous with optimal exponent $\frac{1}{dn+4}$. We note that the optimality of the result was proven much later in \cite{Fischer2016}. Regarding sufficient conditions which ensure that $T_{R_0}>0$, it was first proven by energy methods in \cite{GGDP} that the support does not immediately move towards a given point for $n\in [2,3)$ and $d=1$ if 
\[
          \sup_{r\leq r_0} \, r^{2-d-\frac{8}{n}} \, \int_{B_{R_0+r}\setminus B_{R_0}} \abs{\nabla u_0}^2 \, \mathrm{d}x \,  < +\infty.
\]
This was then extended to dimensions two and three in \cite{GrnHabWTP}. Note that the exponent of $r$ appearing here is the other exponent in the growth condition \eqref{eq-wtp-flatness-condition}. Although not at all obvious, this sufficient condition turns out to not be necessary: It was proven (over 20 years later) in \cite{DeNitti/Fischer} that a condition on the mean value of $u_0$ is also sufficient,
\[
          \sup_{r\leq r_0} \, r^{-d-\frac{4}{n}} \, \int_{B_{R_0+r}\setminus B_{R_0}} u_0(x) \, \mathrm{d}x \,  < +\infty .
\]
Moreover, it was shown in the same work that this condition is optimal in one-dimension. Once again, however, the condition on the mean value is equivalent to the condition on the squared gradient in the case that $u_0$ behaves like a power law near the boundary of the initial support. For an enlightening discussion about these growth conditions, see \cite{DeNitti/Fischer}. 

The existing literature on \eqref{eq-TFPM} with $n\geq 2$ is quite sparse. The only works, to the best of our knowledge, which make a statement about finite speed of propagation are \cite{Cohen,Taranets_Chur_Coating_Cylinder}. In the former, finite speed of propagation and upper bounds on propagation rates are proven in the special case $n=m-1$\footnote{Note that this is disjoint from the range of values we consider throughout this work, as $n\ge 2$ necessitates $m\ge 3$.}. In the latter, a qualitative result on finite speed of propagation is proven under additional assumptions on the relationship between $n$ and $m$. Beyond the special case, there are no results on propagation rates, and for time-delay of support propagation there are no results at all for $n\ge 2$. For $n<2$, the propagation of solutions is better understood due to the works \cite{GSDP,ST}. In \cite{GSDP}, a result analogous to Corollary \ref{col-biggest-ball} is found to be true, while in \cite{ST}, a sufficient condition to induce a time-delay for the support propagation is also derived using entropy estimates. The barrier to results for $n\geq 2$ is the necessity of a closed energy estimate in the form of \eqref{eq-energyest1}. In \cite{Cohen}, this energy estimate is obtained by absorbing the porous-medium terms in \eqref{eq-energyest1} into the more regular thin-film terms, i.e., the signs of the porous-medium terms are not considered when estimating the $H^1$-norm\footnote{This ties back into the discussion of Remark \ref{rem-existence}. Instead of integrating by parts, the term is estimated and absorbed in \cite[Theorem 6]{Cohen}}. This introduces a restriction on the relationship between $n$ and $m$, but in a sort of tradeoff, it allows one to handle values of $m$ greater than two, which we cannot do here. If one narrows their focus to $m\in (1,2)$, no absorption is necessary, because the porous-medium operator is $H^1$-dissipative for $m\in (1,2)$. This fact allows us to separate the two characteristic scaling behaviors belonging to the individual operators, leading to the optimal results found here. 

The results we give here are also a marginal improvement of the classic energy methods. Namely, upper bounds on the rate of propagation for solutions to \eqref{eq-tfe} and \eqref{eq-TFPM}, derived by energy methods, have always come in the form 
\[
    \mathrm{supp} \, u(t) \subset B_{R(t)},
\]
where $R(t)$ is a Hölder continuous function, i.e., in the form of Corollary \ref{col-biggest-ball}. This approach succeeds in describing the overall picture, but there are two important situations which it cannot fully grasp: Firstly, this approach cannot say anything when topology changes can occur, for instance how quickly a hole in the support will be filled; secondly, this approach does not see the shape of the support at all. By estimating $T_w(x_0)$, both of these gaps are filled. 

In view of the prior discussion about sufficient conditions to induce $T_{R_0}>0$, it seems plausible that the condition \eqref{eq-wtp-flatness-condition} in Theorem \ref{thm-main-result-lower} is too strong and that a simple condition on the mean value of $u_0$ may be sufficient to induce a waiting time phenomenon. It is, however, not clear if the methods presented in \cite{DeNitti/Fischer} (which overcome the need for a condition on the gradient of initial data) are applicable to \eqref{eq-TFPM}. Nonetheless, we recover the classical critical growth profiles for polynomial-type initial data, and this condition still captures one of the main curiosities of \eqref{eq-TFPM}, namely that the initial data must be flatter than \emph{both} of the profiles prescribed by the individual operators in order to delay the onset of propagation. This phenomenon (indeed, the exact same critical profile) has been observed for \eqref{eq-TFPM} with $n<2$ in \cite{ST}, and in the aforementioned special case in \cite{Taranets_Chur_Coating_Cylinder}, where $n=m-1$ and $n\in (1,2)$. 

Finally, let us say a brief word about the restriction of the results here to values of $m$ less than two. To treat the case $m\geq 2$, the obstacle to overcome is finding a dissipation estimate for the $H^1$-norm. One foreseeable way to accomplish this is to consider the functional $\mathcal L$, which is given by
\[
    \mathcal L[u(t)] := \frac12 \int_{\R^d} \left[\abs{\nabla u(t)}^2 - \frac{1}{(m-n)(m+1-n)}u(t)^{m+1-n} \right].
\]
A formal computation will yield that $\mathcal L[u(t)]$ is not decreasing in time, but its growth can be controlled on bounded time intervals if $m>\frac23(n+1)$ by using Bernis--Grün estimates. This condition on $m,n$ appears to be the same which is found for proving finite speed of propagation in \cite{Cohen}\footnote{They claim that it even holds for $m>\frac{n+2}{2}$, but we need this condition to verify their argument.}. For this reason, it is reasonable to expect that if one replaces \eqref{eq-energyest1} with a localized estimate for $\mathcal L[u(t)]$, then finite speed of propagation may be proven for $m\geq 2$ as well, under the premise that some conditions on the relationship between $m$ and $n$ are satisfied. However, it is far from clear what should be expected for upper bounds on propagation rates in place of Theorem \ref{thm-main-result-lower}. On the other hand, Theorem \ref{thm-main-result-upper} only explicitly relies on the $\alpha$-entropy estimate, which is attainable for all $m$. Implicitly, there is a reliance on the energy estimate, because the proof of Theorem \eqref{thm-main-result-upper} uses finite speed of propagation. However, if this were known to be true for $m\geq 2$, then we would expect the lower bound on the propagation rate to remain the same.  

\subsection{Layout of the Paper \& Prerequisites}\label{section-proof-strategy}

The proofs of the main results rely on several rather technical estimates and formulas, which we have chosen to prove in separate sections before proving the main results. The paper is structured as follows:

\begin{itemize}
    \item Section \ref{section-lemma1} is dedicated to the proofs of two Stampacchia-type inequalities for the evolutions of the local mass and local $L^2$-norm (s. Lemma \ref{lemma-local-L1} and Lemma \ref{lemma-local-L2}). The resulting inequalities are reminiscent of the energy methods of \cite{BernisFSOP,BernisFSOP2,BasicFSOP,Shishkov/Hulshof,GGDP,AnsiniG}, but with the key difference that they are formulated on balls rather than on infinite cones. To obtain these inequalities, we start by following the approach of \cite{GS} before applying a recently-developed iteration trick (s. Lemma \ref{lemma-iteration-trick}), which was developed to handle similar estimates in \cite{Grun_Sauerbrey_Utley}. 

    \item Section \ref{section-lemma2} is dedicated to the technical results needed to prove Theorem \ref{thm-main-result-upper}, of which there are two. Lemma \ref{lemma-entropy-lower} bounds the entropy growth on solutions from below in terms of the Lebesque measure of the support. Lemma \ref{lemma-entropy-monotonicity} is a proof of a weighted monotonicity formula for the entropy. Both of these results are straightforward generalizations of the celebrated work in \cite{FischerWTU,Fischer_Asymptotic}, which were further refined in \cite{Fischer2016}. Some minor changes must be made to accommodate the porous-medium term.

    \item Section \ref{section-lower} contains the proofs of Theorem \ref{thm-fsop}, Theorem \ref{thm-main-result-lower}, and Corollary \ref{col-biggest-ball}. Proving the first two results simply requires a combination of the results of Section \ref{section-lemma1} with a new formulation of Stampacchia's lemma (Lemma \ref{lemma-stampachhia-wait}), which allows for multiple weights in the Stampacchia-type inequality. The proof of Corollary \ref{col-biggest-ball} is a consequence of Theorem \ref{thm-main-result-lower} combined with some elementary geometric reasoning. 

    \item Section \ref{section-upper} is dedicated to the proof of Theorem \ref{thm-main-result-upper}. This proof succeeds by applying the results of Section \ref{section-lemma2} and following the approach of \cite{Fischer_Asymptotic}. Some novelty arises in the application of this argument, which is a consequence of the changing scaling behavior of $R(x_0;t)$ as derived in Corollary \ref{col-biggest-ball}.

\end{itemize}

\begin{remark}
    Concerning notation: throughout this paper, most of the objects and quantities that we are handling depend on a given point $x\in \R^d$. As we did in the statements of the main results, we will exclude this dependence whenever there is no ambiguity in order to disencumber our notation. In particular, this means that the ball $B_r(x)$ is written as just $B_r$, $R_0(x_0)$ as $R_0$, $T_w(x_0)$ as $T_w$, etc.
\end{remark}

Throughout this work, we make frequent use of the Gagliardo--Nirenberg inequality, which we include here for reference. A proof may be found in \cite[Proposition A.1]{GSDP}.
\begin{lemma} 
\label{lemma-GN}
Let $ 1 \le r \le \infty$, $0< q <p$, $ m \in \N_+$
such that
\[
\frac 1r - \frac mN < \frac 1p \,.
\]
If $\mathcal{O} \subset \R^N$ is bounded with piecewise smooth boundary,
then positive constants $c_1$ and $ c_2$ depending only on $\mathcal{O},
r, p, m$ and $q$ exist such that for any $u \in L^q (\mathcal{O})$
satisfying $ D^m u \in L^r (\mathcal{O})$, the following inequality holds:
\begin{equation}
\| u \|_p \le c_1 \|D^m u\|^a_r \| u \|^{1-a}_q + c_2 \| u \|_q
\label{eq-GN}
\end{equation}
where $a = \frac{\frac 1q - \frac 1p}{\frac 1q + \frac m N - \frac
1r}$.
If, in addition, $\mathcal{O}=\R^N$,  then \eqref{eq-GN} holds
with constant $c_1=c(r,p,m,q)$ and $c_2=0$.
\end{lemma}

\section{Evolution of the Local Mass}\label{section-lemma1}

In this section, we prove a Stampacchia-type estimate on the local mass of solutions to \eqref{eq-TFPM}. This estimate will be the main key to proving Theorem \ref{thm-fsop} and subsequently Theorem \ref{thm-main-result-lower}. We shall also prove an analogous estimate for the local $L^2$-norm, which will be important for the second part of Theorem \ref{thm-main-result-lower} concerning bounds on $T_{R_0}$. In order to prove these estimates, we need a technical iteration result, which will enable us to close the estimates. This technique was developed in \cite{Grun_Sauerbrey_Utley} for the study of finite speed of propagation in stochastic porous-medium equations, and is inspired originally by the ideas in \cite{Shishkov/Hulshof,GGDP}. In contrast to that technique, more than one recursive loop must be carried out in its proof, and this result is much more general than that found in \cite[Lemma B.1]{Grun_Sauerbrey_Utley} or \cite{Shishkov/Hulshof}. In what follows, we denote
\[
    \lltriangle_R =\left\{ (s,\delta) \in (0,R)^2 \,:\, \delta\in (0,R-s) \right\}.
\]

\begin{lemma}
    \label{lemma-iteration-trick} Let $R>0$ and consider non-negative mappings $\mathscr V:[0,R]\ra \R_{\geq0}$, along with non-negative functions $\mathscr E$, $\mathscr A$, $\mathscr U$, and $\mathscr F$ defined on $[0,R]\times \R_+$. If the following conditions are fulfilled:
    \begin{enumerate}
        \item The maps $\mathscr E,\mathscr A,\mathscr U,\mathscr F$ are non-increasing in the first variable;
        \item for any $\epsilon >0$, there exists $C_\epsilon>0$ such that for each $(s,\delta)\in \lltriangle_R$, we have
            \begin{equation}\label{eq-iteration-lemma-general}
                \begin{split}
                    \mathscr V(s+\delta) + \mathscr E(s+\delta,\delta) & \lesssim \, \mathscr A\left(s+\frac{\delta}{2},\frac{\delta}{2}\right) + \mathscr U\left(s+\frac{\delta}{2},\frac{\delta}{2}\right), \\
                    \mathscr U(s+\delta,\delta) & \lesssim \, \epsilon \mathscr U\left(s+\frac{\delta}{2},\frac{\delta}{2}\right) + \epsilon \mathscr E\left(s+\frac{\delta}{2},\frac{\delta}{2} \right) + C_\epsilon \mathscr F\left(s+\frac{\delta}{2}, \frac{\delta}{2} \right);
                \end{split}
            \end{equation}
        \item there exist constants $E,A,U,F\in \R$ such that for any $\eta \in (0,1)$, we have for each $(s,\delta)\in [0,R]\times \R_+$ 
            \begin{equation}\label{eq-iteration-lemma-condition}
                \begin{split}
                    \mathscr E(s, \eta\delta) \leq \eta^E \, \mathscr E(s, \delta), \\
                    \mathscr A(s, \eta\delta) \leq \eta^A \, \mathscr A(s, \delta), \\
                    \mathscr U(s, \eta\delta) \leq \eta^U \, \mathscr U(s, \delta), \\
                    \mathscr F(s, \eta\delta) \leq \eta^F \, \mathscr F(s, \delta);
                \end{split}
            \end{equation}
    \end{enumerate}
    then for any $(s,\delta)\in \lltriangle_R$, we have
    \[
        \mathscr V(s+\delta)  \lesssim \,  \mathscr A\left(s,\delta\right) \, + \, \mathscr F\left(s,\delta\right).
    \]
\end{lemma}

\begin{proof}
    It suffices to prove the result when the multiplicative constants in the inequalities \eqref{eq-iteration-lemma-general} are equal to one. Indeed, by first rescaling $\mathscr A$ and $\mathscr U$ in the first inequality and then rescaling $\mathscr F$ in the second inequality, the multiplicative constants may be taken to be one, and the desired conclusion still holds.

    Consider the second inequality in \eqref{eq-iteration-lemma-general}. Note that the first term appearing on the right is just the same as the term on the left, but with $\delta$ replaced with $\delta/2$. Therefore, we can apply the same estimate again to see that
    \begin{align*}
            \mathscr U(s+\delta,\delta) &\leq \epsilon^2\, \mathscr U\left(s+\frac{\delta}{4},\frac{\delta}{4}\right) + \epsilon\left(\mathscr E\left(s+\frac{\delta}{2},\frac{\delta}{2} \right) + \epsilon \mathscr E\left(s+\frac{\delta}{4},\frac{\delta}{4} \right)\right)   \\
            & \qquad + C_\epsilon \left( \mathscr F\left(s+\frac{\delta}{2}, \frac{\delta}{2} \right) + \epsilon  \mathscr F\left(s+\frac{\delta}{4}, \frac{\delta}{4} \right)\right).
    \end{align*}
    Continuing this process inductively, the following holds for any $n\in \N$:
    \begin{align*}
            \mathscr U\left(s+\delta,\delta\right) &\leq \epsilon^n\, \mathscr U\left(s + \frac{\delta}{2^n},\frac{\delta}{2^n}\right) + \epsilon\left(\sum_{i=0}^{n-1} \epsilon^i \mathscr E\left(s+\frac{\delta}{2^{i+1}},\frac{\delta}{2^{i+1}} \right)\right)   \\
            & \qquad + C_\epsilon \left( \sum_{i=0}^{n-1} \epsilon^i  \mathscr F\left(s+\frac{\delta}{2^{i+1}}, \frac{\delta}{2^{i+1}} \right)\right).
    \end{align*}
    Now using the fact that $\mathscr U,\mathscr E,\mathscr F$ are non-increasing in the first variable along with \eqref{eq-iteration-lemma-condition}, we have 
    \begin{align*}
            \mathscr U\left(s+\delta ,\delta\right) &\leq \epsilon^n 2^{-nU}\, \mathscr U\left(s,\delta\right) + \epsilon\left(\sum_{i=0}^{n-1} \epsilon^i 2^{-iE}\right) \mathscr E\left(s+\frac{\delta}{2},\frac{\delta}{2} \right)  \\
            & \qquad + C_\epsilon \left( \sum_{i=0}^{n-1} \epsilon^i2^{-iF}  \right)\mathscr F\left(s+\frac{\delta}{2}, \frac{\delta}{2} \right).
    \end{align*}
    If we choose $\epsilon$ such that 
    \[
        \epsilon< \min\{ 2^F, 2^E, 2^U \},
    \]
    then we may pass to the limit $n\ra \infty$ in the above inequality to see that 
    \begin{align*}
            \mathscr U\left(s+\delta,\delta\right) &\leq  \, \epsilon \, S_1\, \mathscr E\left(s+\frac{\delta}{2},\frac{\delta}{2} \right) +  C_\epsilon \, S_2\,   \mathscr F\left(s+\frac{\delta}{2}, \frac{\delta}{2} \right).
    \end{align*}
    Here, $S_1,S_2$ are the values of the geometric series obtained from the passage to the limit. Note that if we choose $\epsilon$ to be even smaller, this only decreases the value of $S_1$ and $S_2$. For that reason, we treat them as if they were independent of $\epsilon$. We now use this estimate in the first inequality of \eqref{eq-iteration-lemma-general} to see that
    \begin{align*}
         \mathscr V(s+\delta)+ \mathscr E(s+\delta,\delta) \leq \mathscr A\left(s+\frac{\delta}{2},\frac{\delta}{2}\right) + \epsilon \, S_1\, \mathscr E\left(s+\frac{\delta}{4},\frac{\delta}{4} \right) +  C_\epsilon \, S_2\,   \mathscr F\left(s+\frac{\delta}{4}, \frac{\delta}{4} \right).
    \end{align*}
    Now we go again through the loop, this time iterating over $\mathscr E$. Using an identical argument as before, the following holds for each $n\in \N$:
    \begin{align*}
        \mathscr V(s+\delta)+\mathscr E(s+\delta,\delta) & \leq \left( \sum_{i=0}^{n-1} \left(\epsilon S_1\right)^i\mathscr A\left(s+\frac{\delta}{2\cdot 4^i},\frac{\delta}{2\cdot 4^i}\right) \right) + \left(\epsilon \, S_1\right)^n\, \mathscr E\left(s+\frac{\delta}{4^n},\frac{\delta}{4^n} \right) \\
        & \qquad +  C_\epsilon \, S_2\,\left( \sum_{i=0}^{n-1} \left(\epsilon S_1 \right)^i \, \mathscr F\left(s+\frac{\delta}{4^{i+1}}, \frac{\delta}{4^{i+1}} \right) \right).
    \end{align*}
    Using again the fact that $\mathscr E,\mathscr A$, and $\mathscr F$ are non-increasing in tandem with \eqref{eq-iteration-lemma-condition}, we obtain
    \begin{align*}
        \mathscr V(s+\delta)+ \mathscr E(s+\delta,\delta) & \leq \left( 2^{-A} \sum_{i=0}^{n-1} \left(\epsilon S_1\right)^i\, 4^{-iA}  \right)\mathscr A\left(s,\delta\right) + \left(\epsilon \, S_1\right)^n\, 4^{-nE} \, \mathscr E\left(s,\delta\right) \\
        & \qquad +  C_\epsilon \, S_2\,\left( 4^{-F}\sum_{i=0}^{n-1} \left(\epsilon S_1 \right)^i 4^{-iF}  \right)\mathscr F\left(s, \delta \right).
    \end{align*}
    Choosing once more $\epsilon$ small enough and passing to the limit $n\ra \infty$, we have 
    \begin{align*}
        \mathscr V(s+\delta)+ \mathscr E(s+\delta,\delta) & \leq S_3\, \mathscr A\left(s,\delta\right) +  C_\epsilon \, S_2\,S_4\,\mathscr F\left(s, \delta \right).
    \end{align*}
    
\end{proof}

With this result in mind, we proceed to the integral estimates. The first controls the local mass:
\begin{lemma}\label{lemma-local-L1}
    Let $u$ be a strongly energy dissipating solution satisfying \eqref{ass-u0} and \eqref{ass-nm}. There exists a constant $C(n,m,d)>0$ such that for any $x\in \R^d$, the following holds for any $0\leq t_0<t_1$ and any $s,\delta$ such that $0< s<s+\delta$:
    \begin{equation}
        \label{eq-local-L1}
        \begin{split}
        \mathcal L^d(B_s)\inv \, \left(\sup_{t_0\leq t\leq t_1} \, \int_{B_{s}} u(t)\right)^2 \,  & \lesssim_{(n,m,d)} \, \int_{B_{s+\delta}} u(t_0)^2 \,  + \, \delta^2 \int_{B_{s+\delta}}\abs{\nabla u(t_0)}^2 \\
        & \hspace{1cm} + \, (t_1-t_0) \, \delta^{-(dm+2)} \left(\sup_{t_0\leq t\leq t_1} \, \int_{B_{s+\delta}} u(t) \right)^{m+1} \\
        & \hspace{1cm }+ \, (t_1-t_0) \, \delta^{-(d(n+1)+4))}\left( \sup_{t_0\leq t\leq t_1} \, \int_{B_{s+\delta}} u(t) \right)^{n+2}.
        \end{split}
    \end{equation}
\end{lemma}

\begin{proof}
    The proof uses a combination of the weighted energy estimate \eqref{eq-energyest1} with a weighted $L^2$-estimate, followed by an application of the two-step iteration of Lemma \ref{lemma-iteration-trick}.

    Fix now $0\leq t_0<t_1$ and $x\in \R^d$. From here on, all balls $B_r$ are assumed to be centered at $x$. For an arbitrary choice of $s,\delta>0$ such that $0< s<s+\delta$, we consider localizing functions $\varphi_{s,\delta}\in C^2_c(\R^d)$ with the following properties: First, $\varphi_{s,\delta}$ takes values in $[0,1]$ such that $\varphi_{s,\delta} = 1$ identically in $B_s$ and $\varphi_{s,\delta}= 0$ identically in $\R^d\setminus B_{s+\delta}$. Second, there exists a constant $c>0$ that does not depend on $s$ or $\delta$ such that the following holds:
    \[
        \abs{\nabla\varphi_{s,\delta}} \leq \frac{c}{\delta}, \quad \abs{D^2\varphi_{s,\delta}} \leq \frac{c}{\delta^2}.
    \]
    If, for an arbitrary but fixed choice of $s,\delta$, we use $\zeta = \varphi_{s,\frac{\delta}{2}}$ in the energy estimate \eqref{eq-energyest1}, then we obtain 
    \begin{equation}
    \begin{split}
        \label{eq-lemma-2.1-loc-en-1}
        \int_{\R^d} \varphi_{s,\frac{\delta}{2}}^6\abs{\nabla u(t)}^2 \, \bigg\vert_{t_0}^{t_1} & \, +  \, \int_{t_0}^{t_1}\int_{\R^d} \varphi_{s,\frac{\delta}{2}}^6 \left[ \abs{\nabla u^{\frac{n+2}{6}}}^6 + \abs{\nabla\Delta u^{\frac{n+2}{2}}}^2 + \abs{\nabla u^{\frac{m+1}{4}}}^4
         \right] \\
         & \hspace{0.5cm }\lesssim_{(n,m)} \,  \int_{B_{s+\frac{\delta}{2}}}\abs{\nabla u_0}^2 \, + \,  \left(\frac{2}{\delta}\right)^{6}\int_{t_0}^{t_1}\int_{B_{s+\frac{\delta}{2}}} u^{n+2} \, + \, \left(\frac{2}{\delta}\right)^{4}\int_{t_0}^{t_1}\int_{B_{s+\frac{\delta}{2}}} u^{m+1}.
         \end{split}
    \end{equation}
    Hence we have 
    \begin{equation}
        \begin{split}
        \label{eq-lemma-2.1-loc-en-2}
        &\int_{t_0}^{t_1}\int_{\R^d} \varphi_{s,\frac{\delta}{2}}^6 \left[ \abs{\nabla u^{\frac{n+2}{6}}}^6 + \abs{\nabla\Delta u^{\frac{n+2}{2}}}^2 \, + \abs{\nabla u^{\frac{m+1}{4}}}^4 
         \right] \\
         & \hspace{2cm}\lesssim_{(n,m)} \, \left(\frac{2}{\delta}\right)^{6}\left[ \left( \frac{\delta}{2}\right)^6\int_{B_{s+\frac{\delta}{2}}} \abs{\nabla u(t_0)}^2 \, + \, \int_{t_0}^{t_1}\int_{B_{s+\frac{\delta}{2}}} u^{n+2} \, + \, \left(\frac{\delta}{2}\right)^{2}\int_{t_0}^{t_1}\int_{B_{s+\frac{\delta}{2}}} u^{m+1}\right].
         \end{split}
    \end{equation}
    We shall use the previous estimate to close a localized $L^2$-estimate. To this end, we note first that $\varphi_{s,\frac{\delta}{2}}^6u$ is a valid test function for \eqref{eq-weakform}, hence (after a straightforward density argument) we have
    \begin{equation}
        \label{eq-lemma-2.1-loc-l2-1}
        \begin{split}
         \int_{\R^d}& \varphi_{s,\frac{\delta}{2}}^6\, u(t)^2 \bigg\vert_{t_0}^{t_1} \,= \, - \int_{t_0}^{t_1}\int_{\R^d} \varphi_{s,\frac{\delta}{2}}^6 \, u^{m-1}\abs{\nabla u}^2 + \frac{1}{m+1}\int_{t_0}^{t_1}\int_{\R^d} \Delta\left( \varphi_{s,\frac{\delta}{2}}^6\right) \, u^{m+1} \\
        &\hspace{1cm} + \, \int_{t_0}^{t_1}\int_{\R^d} \nabla\left( \varphi_{s,\frac{\delta}{2}}^6 u \right)\cdot u^n\nabla\Delta u \\
        &\lesssim_{(n,m)} - \int_{t_0}^{t_1}\int_{\R^d} \varphi_{s,\frac{\delta}{2}}^6 \, \abs{\nabla u^{\frac{m+1}{2}}}^2 + \int_{t_0}^{t_1}\int_{\R^d} \Delta\left( \varphi_{s,\frac{\delta}{2}}^6\right) \, u^{m+1} \\
        &\hspace{1cm} + \, \left(\int_{t_0}^{t_1}\int_{\R^d} \varphi_{s,\frac{\delta}{2}}^6  \abs{\nabla u^{\frac{n+2}{6}}}^6\right)^{\frac{1}{6}}\left( \int_{t_0}^{t_1}\int_{\R^d}\varphi_{s,\frac{\delta}{2}}^6 \abs{\nabla\Delta u^{\frac{n+2}{2}}}^2\right)^{\frac
        12}\left( \int_{t_0}^{t_1}\int_{\R^d}  \varphi_{s,\frac{\delta}{2}}^6 \, u^{n+2}\right)^{\frac{1}{3}} \\
        &\hspace{1cm} + \, \left( \int_{t_0}^{t_1}\int_{\R^d} \varphi_{s,\frac{\delta}{2}}^6 \abs{\nabla\Delta u^{\frac{n+2}{2}}}^2\right)^{\frac
        12}\left( \int_{t_0}^{t_1}\int_{\R^d}  \varphi_{s,\frac{\delta}{2}}^4\abs{\nabla\varphi_{s,\frac{\delta}{2}}}^2 \, u^{n+2}\right)^{\frac{1}{2}} \\
        & \lesssim \,  - \int_{t_0}^{t_1}\int_{\R^d} \varphi_{s,\frac{\delta}{2}}^6 \, \abs{\nabla u^{\frac{m+1}{2}}}^2 + \left( \frac{2}{\delta}\right)^2 \int_{t_0}^{t_1}\int_{B_{s+\frac{\delta}{2}}} u^{m+1} \\
        &\hspace{1cm} + \, \left(\int_{t_0}^{t_1}\int_{\R^d} \varphi_{s,\frac{\delta}{2}}^6  \abs{\nabla u^{\frac{n+2}{6}}}^6\right)^{\frac{1}{6}}\left( \int_{t_0}^{t_1}\int_{\R^d}\varphi_{s,\frac{\delta}{2}}^6 \abs{\nabla\Delta u^{\frac{n+2}{2}}}^2\right)^{\frac
        12}\left( \int_{t_0}^{t_1}\int_{B_{s+\frac{\delta}{2}}} u^{n+2}\right)^{\frac{1}{3}} \\
        &\hspace{1cm} + \, \left(\frac{2}{\delta} \right)\left( \int_{t_0}^{t_1}\int_{\R^d} \varphi_{s,\frac{\delta}{2}}^6 \abs{\nabla\Delta u^{\frac{n+2}{2}}}^2\right)^{\frac
        12}\left( \int_{t_0}^{t_1}\int_{B_{s+\frac{\delta}{2}}}  u^{n+2}\right)^{\frac{1}{2}}.
        \end{split}
    \end{equation}
    On the right hand side of \eqref{eq-lemma-2.1-loc-l2-1}, we see terms which also appear on the left hand side of \eqref{eq-lemma-2.1-loc-en-2}. Furthermore, the remaining integrals of $u^{n+2}$ may be estimated from above by the entire bracketed term in \eqref{eq-lemma-2.1-loc-en-2}. Putting these observations together yields 
    \begin{equation}
        \label{eq-lemma-2.1-loc-l2-2}
        \begin{split}
        \int_{\R^d} \varphi_{s,\frac{\delta}{2}}^6\, u(t)^2 \bigg\vert_{t_0}^{t_1} \,
        & \lesssim_{(n,m)} \,  - \int_{t_0}^{t_1}\int_{\R^d} \varphi_{s,\frac{\delta}{2}}^6 \, \abs{\nabla u^{\frac{m+1}{2}}}^2 + \left( \frac{2}{\delta}\right)^2 \int_{t_0}^{t_1}\int_{B_{s+\frac{\delta}{2}}} u^{m+1} \\
        &\hspace{1cm} + \, \left( \frac{2}{\delta}\right)^4 \left[ \left( \frac{\delta}{2}\right)^6\int_{B_{s+\frac{\delta}{2}}} \abs{\nabla u(t_0)}^2 \, + \, \int_{t_0}^{t_1}\int_{B_{s+\frac{\delta}{2}}} u^{n+2} \, + \, \left(\frac{\delta}{2}\right)^{2}\int_{t_0}^{t_1}\int_{B_{s+\frac{\delta}{2}}} u^{m+1}\right] \\
        & \lesssim_{(n,m)} \,  - \int_{t_0}^{t_1}\int_{\R^d} \varphi_{s,\frac{\delta}{2}}^6 \, \abs{\nabla u^{\frac{m+1}{2}}}^2 + \left( \frac{2}{\delta}\right)^2 \int_{t_0}^{t_1}\int_{B_{s+\frac{\delta}{2}}} u^{m+1} \\
        &\hspace{2cm} + \, \left( \frac{\delta}{2}\right)^2\int_{B_{s+\frac{\delta}{2}}} \abs{\nabla u(t_0)}^2 \, + \,\left( \frac{2}{\delta}\right)^4 \int_{t_0}^{t_1}\int_{B_{s+\frac{\delta}{2}}} u^{n+2}.
        \end{split}
    \end{equation}
    As a consequence of \eqref{eq-lemma-2.1-loc-en-2}, we see that 
    \begin{equation}
        \label{eq-lemma-2.1-loc-en-3}
        \begin{split}
        &\left( \frac{\delta}{2} \right)^2\int_{t_0}^{t_1}\int_{\R^d} \varphi_{s,\frac{\delta}{2}}^6 \left[ \abs{\nabla u^{\frac{n+2}{6}}}^6 + \abs{\nabla\Delta u^{\frac{n+2}{2}}}^2 \, + \abs{\nabla u^{\frac{m+1}{4}}}^4 
         \right] \\
         & \hspace{2cm}\lesssim_{(n,m)} \, \left(\frac{2}{\delta}\right)^{4}\left[ \left( \frac{\delta}{2}\right)^6\int_{B_{s+\frac{\delta}{2}}} \abs{\nabla u(t_0)}^2 \, + \, \int_{t_0}^{t_1}\int_{B_{s+\frac{\delta}{2}}} u^{n+2} \, + \, \left(\frac{\delta}{2}\right)^{2}\int_{t_0}^{t_1}\int_{B_{s+\frac{\delta}{2}}} u^{m+1}\right].
         \end{split}
    \end{equation}
    Adding together the inequalities \eqref{eq-lemma-2.1-loc-l2-2} and \eqref{eq-lemma-2.1-loc-en-3}, we obtain the following:
    \begin{equation}
        \label{eq-lemma-2.1-iteration-step1}
        \begin{split}
        \int_{\R^d} \varphi_{s,\frac{\delta}{2}}^6\, u(t_1)^2 \, & + \, \left(\frac{\delta}{2} \right)^2\int_{t_0}^{t_1}\int_{\R^d} \varphi_{s,\frac{\delta}{2}}^6 \left[ \abs{\nabla u^{\frac{n+2}{6}}}^6 + \abs{\nabla u^{\frac{m+1}{4}}}^4 
         \right] \\
         & \lesssim_{(n,m)} \, \int_{B_{s+\frac{\delta}{2}}} u(t_0)^2 \, + \, \left( \frac{\delta}{2}\right)^2\int_{B_{s+\frac{\delta}{2}}} \abs{\nabla u(t_0)}^2 \,\\ 
         &\hspace{1cm} + \, \left( \frac{2}{\delta}\right)^2 \int_{t_0}^{t_1}\int_{B_{s+\frac{\delta}{2}}} u^{m+1} 
         \, + \,\left( \frac{2}{\delta}\right)^4 \int_{t_0}^{t_1}\int_{B_{s+\frac{\delta}{2}}} u^{n+2}.
         \end{split}
    \end{equation}
    Since $t_1>t_0$ are arbitrary, and recalling that $\varphi_{s,\frac{\delta}{2}}=1$ identically in $B_s$, we may just as well write
    \begin{equation}
        \label{eq-lemma-2.1-iteration-step2}
        \begin{split}
        \sup_{t_0\leq t\leq t_1}\int_{B_s} u(t)^2 \, & + \, \left(\frac{\delta}{2} \right)^2\int_{t_0}^{t_1}\int_{B_{s}} \left[ \abs{\nabla u^{\frac{n+2}{6}}}^6 + \abs{\nabla u^{\frac{m+1}{4}}}^4 
         \right] \\
         & \lesssim_{(n,m)} \, \int_{B_{s+\frac{\delta}{2}}} u(t_0)^2 \, + \, \left( \frac{\delta}{2}\right)^2\int_{B_{s+\frac{\delta}{2}}} \abs{\nabla u(t_0)}^2 \,\\ 
         &\hspace{1cm} + \, \left( \frac{2}{\delta}\right)^2 \int_{t_0}^{t_1}\int_{B_{s+\frac{\delta}{2}}} u^{m+1} 
         \, + \,\left( \frac{2}{\delta}\right)^4 \int_{t_0}^{t_1}\int_{B_{s+\frac{\delta}{2}}} u^{n+2}.
         \end{split}
    \end{equation}
    This inequality will eventually yield the first inequality of \eqref{eq-iteration-lemma-general}. The last two integrals on the right are estimated using the Gagliardo--Nirenberg interpolation inequality (Lemma \ref{lemma-GN}). For the first integral, note that for all $t\in (t_0,t_1)$, one has:
    \begin{equation}
        \label{eq-lemma-2.1-GN-m-term}
        \begin{split}
        \delta^{-2}\int_{B_{s}} u(t)^{m+1} \, & \leq \, \delta^{-2} \int_{\R^d}\varphi_{s,\frac{\delta}{2}}^4 \, u(t)^{m+1} \\
        & \lesssim_{(m,d)} \, \delta^{-2}\left( \int_{\R^d}\abs{\nabla\left(\varphi_{s,\frac{\delta}{2}} \, u(t)^{\frac{m+1}{4}}\right)}^4\right)^{\vartheta}\left( \int_{\R^d} \varphi_{s,\frac{\delta}{2}}^{\frac{4}{m+1}} \, u(t) \right)^{(1-\vartheta)(m+1)}\\
        & \lesssim_{(m,d)} \, \epsilon\left(\frac{\delta}{2}\right)^2\int_{\R^d}\abs{\nabla\left(\varphi_{s,\frac{\delta}{2}} \, u(t)^{\frac{m+1}{4}}\right)}^4 + C_\epsilon \left( \frac{2}{\delta} \right)^{\frac{2+2\vartheta}{1-\vartheta}} \left( \int_{B_{s+\frac{3\delta}{4}}} u(t) \right)^{m+1}\\
        & \lesssim_{(m,d)} \, \epsilon\left(\frac{\delta}{2}\right)^2\int_{B_{s+\frac{\delta}{2}}}\abs{\nabla u(t)^{\frac{m+1}{4}}}^4 \, +\, \epsilon \left( \frac{2}{\delta} \right)^2\int_{B_{s+\frac{\delta}{2}}} u(t)^{m+1} \, \\
        & \hspace{2cm}+ \, C_\epsilon \left( \frac{2}{\delta} \right)^{\frac{2+2\vartheta}{1-\vartheta}} \left( \int_{B_{s+\frac{\delta}{2}}} u(t) \right)^{m+1}.
        \end{split}
    \end{equation}
    Here we have 
    \[
        \vartheta = \frac{dm}{4+dm},
    \]
    which allows us to compute
    \[
        \frac{2+2\vartheta}{1-\vartheta} = dm+2.
    \]
    Integrating \eqref{eq-lemma-2.1-GN-m-term} from $t_0$ to $t_1$ gives then 
    \begin{equation}
        \label{eq-lemma-2.1-iteration-setup-3}
                \begin{split}
        \delta^{-2}\int_{t_0}^{t_1}\int_{B_{s}} u^{m+1} \,
        & \lesssim_{(m,d)} \, \epsilon\left(\frac{\delta}{2}\right)^2\int_{t_0}^{t_1}\int_{B_{s+\frac{\delta}{2}}}\abs{\nabla u^{\frac{m+1}{4}}}^4 \, +\, \epsilon \left( \frac{2}{\delta} \right)^2\int_{t_0}^{t_1}\int_{B_{s+\frac{\delta}{2}}} u^{m+1} \, \\
        & \hspace{2cm}+ \, C_\epsilon \, (t_1-t_0) \,\left( \frac{2}{\delta} \right)^{dm+2}\, \left( \sup_{t_0\leq t\leq t_1}\int_{B_{s+\frac{\delta}{2}}} u(t) \right)^{m+1}.
        \end{split}
    \end{equation}
    We perform an analogous computation for the second integral on the right hand side of \eqref{eq-lemma-2.1-iteration-step2}. For all $t\in (t_0,t_1)$:
     \begin{equation}
        \label{eq-lemma-2.1-GN-n-term}
        \begin{split}
        \delta^{-4}\int_{B_{s+\frac{\delta}{2}}} u(t)^{n+2} \, & \leq \, \delta^{-4} \int_{\R^d}\varphi_{s,\frac{\delta}{2}}^6 \, u(t)^{n+2} \\
        & \lesssim_{(n,d)} \, \delta^{-4}\left( \int_{\R^d}\abs{\nabla\left(\varphi_{s,\frac{\delta}{2}} \, u(t)^{\frac{n+2}{6}}\right)}^6\right)^{\theta}\left( \int_{\R^d} \varphi_{s,\frac{\delta}{2}}^{\frac{6}{n+2}} \, u(t) \right)^{(1-\theta)(n+2)}\\
        & \lesssim_{(n,d)} \, \epsilon\left(\frac{\delta}{2}\right)^4\int_{\R^d}\abs{\nabla\left(\varphi_{s,\frac{\delta}{2}} \, u(t)^{\frac{n+2}{6}}\right)}^6 + C_\epsilon \left( \frac{2}{\delta} \right)^{\frac{4+2\theta}{1-\theta}} \left( \int_{B_{s+\frac{\delta}{2}}} u(t) \right)^{n+2}\\
        & \lesssim_{(m,d)} \, \epsilon\left(\frac{\delta}{2}\right)^4\int_{B_{s+\frac{3\delta}{4}}}\abs{\nabla u(t)^{\frac{n+2}{6}}}^6 \, +\, \epsilon \left( \frac{2}{\delta} \right)^2\int_{B_{s+\frac{\delta}{2}}} u(t)^{n+2} \, \\
        & \hspace{2cm}+ \, C_\epsilon \left( \frac{2}{\delta} \right)^{\frac{4+2\theta}{1-\theta}} \left( \int_{B_{s+\frac{\delta}{2}}} u(t) \right)^{n+2}.
        \end{split}
    \end{equation}
    In this case
    \[
        \theta = \frac{d(n+1)}{6+d(n+1)},
    \]
    from which we derive
    \[
        \frac{4+2\theta}{1-\theta} = 4+d(n+1).
    \]
    Integrating \eqref{eq-lemma-2.1-GN-n-term} from $t_0$ to $t_1$ yields 
    \begin{equation}
        \label{eq-lemma-2.1-iteration-setup-4}
                \begin{split}
        \delta^{-4}\int_{t_0}^{t_1}\int_{B_{s}} u^{n+2} \,
        & \lesssim_{(n,d)} \, \epsilon\left(\frac{\delta}{2}\right)^2\int_{t_0}^{t_1}\int_{B_{s+\frac{\delta}{2}}}\abs{\nabla u^{\frac{n+2}{6}}}^6 \, +\, \epsilon \left( \frac{4}{\delta} \right)^4\int_{t_0}^{t_1}\int_{B_{s+\frac{\delta}{2}}} u^{n+2} \, \\
        & \hspace{2cm}+ \, C_\epsilon \, (t_1-t_0) \,\left( \frac{2}{\delta} \right)^{d(n+1)+4}\, \left( \sup_{t_0\leq t\leq t_1}\int_{B_{s+\frac{\delta}{2}}} u(t) \right)^{n+2}.
        \end{split}
    \end{equation}
    Adding the estimates estimates \eqref{eq-lemma-2.1-iteration-setup-3} and \eqref{eq-lemma-2.1-iteration-setup-4}, we have 
    \begin{equation}\label{eq-lemma-2.1-iteration-master2}\begin{split}
          \delta^{-2}\int_{t_0}^{t_1}\int_{B_{s}} u^{m+1} \,&  + \, \delta^{-4}\int_{t_0}^{t_1}\int_{B_{s}} u^{n+2} \, \\
            &\lesssim_{(n,d)} \, \epsilon\left(\frac{\delta}{2}\right)^2\int_{t_0}^{t_1}\int_{B_{s+\frac{\delta}{2}}}\abs{\nabla u^{\frac{n+2}{6}}}^6 \, +\, \epsilon \left( \frac{2}{\delta} \right)^4\int_{t_0}^{t_1}\int_{B_{s+\frac{\delta}{2}}} u^{n+2} \, \\
            &\qquad  + \, \epsilon\left(\frac{\delta}{2}\right)^2\int_{t_0}^{t_1}\int_{B_{s+\frac{\delta}{2}}}\abs{\nabla u^{\frac{m+1}{4}}}^4 \, +\, \epsilon \left( \frac{2}{\delta} \right)^2\int_{t_0}^{t_1}\int_{B_{s+\frac{\delta}{2}}} u^{m+1} \, \\
            & \hspace{2cm}+ \, C_\epsilon \, (t_1-t_0) \,\left( \frac{2}{\delta} \right)^{dm+2}\, \left( \sup_{t_0\leq t\leq t_1}\int_{B_{s+\frac{\delta}{2}}} u(t) \right)^{m+1} \\
            &\hspace{2cm}+ \, C_\epsilon \, (t_1-t_0) \,\left( \frac{2}{\delta} \right)^{d(n+1)+4}\, \left( \sup_{t_0\leq t\leq t_1}\int_{B_{s+\frac{\delta}{2}}} u(t) \right)^{n+2}.
    \end{split}
    \end{equation}
    We are now prepared to apply Lemma \ref{lemma-iteration-trick}. To that end, let $R>0$ be arbitrary. Define 
    \begin{align*}
        \mathscr V(s) & = \sup_{t_0\leq t\leq t_1}\int_{B_{R-s}} u(t)^2, \\
        \mathscr E(s,\delta) & = \delta^2\int_{t_0}^{t_1}\int_{B_{R-s}} \left[ \abs{\nabla u^{\frac{n+2}{6}}}^6 + \abs{\nabla u^{\frac{m+1}{4}}}^4 
         \right], \\
         \mathscr U(s,\delta) & = \delta^{-2}\int_{t_0}^{t_1}\int_{B_{R-s}} u^{m+1} \, + \, \delta^{-4}\int_{t_0}^{t_1}\int_{B_{R-s}} u^{n+2}, \\
         \mathscr F(s,\delta) & = (t_1-t_0) \,\delta^{-(dm+2)}\, \left( \sup_{t_0\leq t\leq t_1}\int_{B_{R-s}} u(t) \right)^{m+1} \\
            &\hspace{2cm}+ (t_1-t_0) \,\delta^{-(d(n+1)+4)}\, \left( \sup_{t_0\leq t\leq t_1}\int_{B_{R-s}} u(t) \right)^{n+2}, \\
            \mathscr A(s,\delta) &= \int_{B_{R-s}} u(t_0)^2 \, + \, \delta^{2}\int_{B_{R-s}} \abs{\nabla u(t_0)}^2.
    \end{align*}
    Re-writing the inequalities \eqref{eq-lemma-2.1-iteration-step2} and \eqref{eq-lemma-2.1-iteration-master2} in terms of these functions, we see that all of the conditions to apply Lemma \ref{lemma-iteration-trick} are met. Therefore, for all $(s,\delta)\in \lltriangle_R$ we have 
    \begin{equation}
        \label{eq-lemma-2.1-iteration-final}
        \begin{split}
        \sup_{t_0\leq t\leq t_1}\int_{B_{R-s-\delta}} u(t)^2 \,
         &\lesssim_{(n,m,d)} \, \int_{B_{R-s}} u(t_0)^2 \, + \, \delta^2\int_{B_{R-s}} \abs{\nabla u(t_0)}^2 \\ 
         &\hspace{1cm} + \,\left(t_1-t_0 \right) \left[\delta^{-(dm+2)}\, \left( \sup_{t_0\leq t\leq t_1}\int_{B_{R-s}} u(t) \right)^{m+1} \, \right. \\
         &\hspace{3.5cm}\left.+ \, \delta^{-(d(n+1)+4)}\, \left( \sup_{t_0\leq t\leq t_1}\int_{B_{R-s}} u(t) \right)^{n+2}  \right].
         \end{split}
    \end{equation}
    Since $R$ was arbitrary, the desired inequality \eqref{eq-local-L1} follows after using Hölder's inequality to estimate
    \[
        \int_{B_s} u(t)^2 \, \geq \, \mathcal{L}^d\left(B_s\right)\inv \left( \int_{B_s} u(t)\right)^2.
    \]
    
\end{proof}

\begin{lemma}\label{lemma-local-L2}
    Let $u$ be a strongly energy dissipating solution satisfying \eqref{ass-u0} and \eqref{ass-nm}. There exists a constant $C=C(n,m,d)>0$ such that for any $x\in \R^d$, the following holds for any $T\geq 0$ and any $s,\delta$ such that $0< s<s+\delta$:
    \begin{equation}
        \label{eq-local-L2}
        \begin{split}
       \sup_{t\leq T} \, \int_{B_{s}} u(t)^2 \,  & \lesssim_{(n,m,d)} \, \int_{B_{s+\delta}} u_0^2 \,  + \, \delta^2 \int_{B_{s+\delta}}\abs{\nabla u_0}^2 \\
        & \hspace{1cm} + \, T\delta^{-2-\frac{d(m-1)}{2}} \left(\sup_{t\leq T} \, \int_{B_{s+\delta}} u(t)^2 \right)^{\frac{m+1}{2}} + \, T\delta^{-4-\frac{dn}{2}}\left( \sup_{t\leq T} \, \int_{B_{s+\delta}} u(t)^2 \right)^{\frac{n+2}{2}}.
        \end{split}
    \end{equation}
\end{lemma}

\begin{proof}
    Identical to the proof of the previous lemma, except that one interpolates with the $L^2$-norm instead of the $L^1$-norm in the steps \eqref{eq-lemma-2.1-GN-m-term} and \eqref{eq-lemma-2.1-GN-n-term}.
    
\end{proof}

\section{Lower Bounds on Entropy Production}\label{section-lemma2}

In this section, we prove that for strongly energy dissipating solutions to \eqref{eq-TFPM}, an upper bound on 
\[
    \mathcal L^d\left(  \mathrm{supp} \, u(t)\right), \quad t\in (0,t_0]
\]
implies a lower bound for the entropy at the time $t_0$:
\[
    \int_{\R^d} \, u(t_0)^{\alpha+1}.
\]
This approach was used first to find lower bounds on the propagation rate for thin-film equations in \cite{Fischer_Asymptotic}. We will succeed in adapting this technique to our purposes, but an extra condition will arise when $d=3$. The coming proof of Theorem \ref{thm-main-result-upper} also uses the same approach as is found in  \cite{Fischer_Asymptotic}, and so it will also be necessary to derive a weighted monotonicity formula for the entropy. That is the content of the second result of this section. 

\begin{lemma}\label{lemma-entropy-lower}
    Let $d\in \{1,2\}$ and $u$ be a strongly energy dissipating solution satisfying \eqref{ass-u0}, \eqref{ass-nm}, and such that $u$ has finite speed of propagation. For any $\alpha \in \left(\max\{\frac{1}{2}-n,-m,-1\}, 2-n\right)$, there exists a constant $C=C(\alpha,n,m,d)$ such that for any $t_0 >0$, we have 
    \begin{equation}\label{eq-entropy-lower}\begin{split}
        \int_{\R^d} \, u(t_0)^{\alpha+1} \, & \gtrsim_{(\alpha,n,m,d)} \,   \norm{u_0}_{L^1}^{\alpha+n+1} \int_0^{t_0} \mathcal{L}^d\left( \mathrm{supp}\, u(t) \right)^{-\frac{d(\alpha+n)+4}{d}} \mathrm{d}t  \\
        &\hspace{3.5cm} + \norm{u_0}_{L^1}^{\alpha+m} \int_0^{t_0} \mathcal{L}^d\left( \mathrm{supp}\, u(t) \right)^{-\frac{d(\alpha+m-1)+2}{d}} \, \mathrm{d}t  .\end{split}
    \end{equation}
    Moreover, if the further condition $\alpha > \frac{1}{3}-m$ is imposed, the above inequality holds for $d=3$ as well. 
\end{lemma}

\begin{proof}
    This follows as a consequence of the $\alpha$-entropy estimate \eqref{eq-alphaentropy} and the Gagliardo--Nirenberg inequality \eqref{eq-GN}. This proof follows \cite[Lemma 12]{Fischer_Asymptotic}. By \eqref{eq-alphaentropy}, we have for any $\alpha \in \left(\max\{\frac{1}{2}-n,-m,-1\},2-n\right) \cap (-1,0)$ and $t_0\geq 0$:
    \begin{equation}
        \label{eq-lemma-3.1-initial}
        \int_{\R^d} \, u(t_0)^{\alpha+1} \, \gtrsim_{(\alpha,n,m)} \, \int_0^{t_0}\int_{\R^d} \abs{\nabla u^{\frac{\alpha+n+1}{4}}}^4 + \int_0^{t_0}\int_{\R^d} \abs{\nabla u^{\frac{\alpha+m}{2}}}^2.
    \end{equation}
    Note here that we are already using the property of finite speed of propagation to justify using $\zeta = 1$ in \eqref{eq-alphaentropy}. Now, the integrals on the right may be estimated in terms of the mass and the Lebesgue measure of the support via the following estimate, which holds for all $t>0$:
    \begin{equation}
        \label{eq-lemma-3.1-GN-nterm}\begin{split}
        \norm{u_0}_{1}^{\alpha+n+1} & = \norm{u(t)}_{1}^{\alpha+n+1}  \, \leq \, \mathcal{L}^d\left( \mathrm{supp}\, u(t) \right)^{\alpha+n} \, \norm{u(t)}_{\alpha+n+1}^{\alpha+n+1} \\
        &\lesssim_{(n,\alpha,d)} \,  \mathcal{L}^d\left( \mathrm{supp}\, u(t) \right)^{\alpha+n} \, \norm{\nabla u^{\frac{\alpha+n+1}{4}}}_4^{4\theta} \, \norm{u_0}_{1}^{(1-\theta)(\alpha+n+1)}.
        \end{split}
    \end{equation}
    Here we have applied the Gagliardo--Nirenberg inequality \eqref{eq-GN} with 
    \[
        \theta = \frac{d(\alpha+n)}{d(\alpha+n)+4}.
    \]
    We have also used the fact that $\alpha+n+1>1$ for all $\alpha\in (-1,0)$. Rearranging \eqref{eq-lemma-3.1-GN-nterm} to solve for the gradient term before integrating in time yields
    \begin{equation}
        \label{eq-lemma-3.1-nterm-done}
        \int_0^{t_0}\int_{\R^d} \abs{\nabla u^{\frac{\alpha+n+1}{4}}}^4 \,  \gtrsim_{(n,\alpha,d)} \, \norm{u_0}_1^{\alpha+n+1} \int_0^{t_0} \mathcal{L}^d\left( \mathrm{supp}\, u(t) \right)^{-\frac{d(\alpha+n)+4}{d}} \, \mathrm{d}t.
    \end{equation}
    To estimate the second integral on the right hand side of \eqref{eq-lemma-3.1-initial}, we first consider the case in which $\alpha+m>1$. In that event, the computation is analogous, and Hölder's inequality and the Gagliardo--Nirenberg inequality yield 
    \begin{equation}
        \label{eq-lemma-3.1-GN-mterm}\begin{split}
        \norm{u_0}_{1}^{\alpha+m} & = \norm{u(t)}_{1}^{\alpha+m}  \, \leq \, \mathcal{L}^d\left( \mathrm{supp}\, u(t) \right)^{\alpha+m-1} \, \norm{u(t)}_{\alpha+m}^{\alpha+m} \\
        &\lesssim_{(m,\alpha,d)} \,  \mathcal{L}^d\left( \mathrm{supp}\, u(t) \right)^{\alpha+m-1} \, \norm{\nabla u^{\frac{\alpha+m}{2}}}_2^{2\xi} \, \norm{u_0}_{1}^{(1-\xi)(\alpha+m)},
        \end{split}
    \end{equation}
    where 
    \[
        \xi = \frac{d(\alpha+m-1)}{d(\alpha+m-1)+2}.
    \]
    Solving \eqref{eq-lemma-3.1-GN-mterm} now for the gradient term before integrating in time yields
    \begin{equation}
        \label{eq-lemma-3.1-mterm-done}
        \int_0^{t_0}\int_{\R^d} \abs{\nabla u^{\frac{\alpha+m}{2}}}^2 \,  \gtrsim_{(m,\alpha,d)} \, \norm{u_0}_1^{\alpha+m} \int_0^{t_0} \mathcal{L}^d\left( \mathrm{supp}\, u(t) \right)^{-\frac{d(\alpha+m-1)+2}{d}} \, \mathrm{d}t.
    \end{equation}
    In the case that $\alpha+m < 1$, the first inequality in \eqref{eq-lemma-3.1-GN-mterm} is no longer valid. To amend this issue, let $p>0$ be such that 
    \[
        p> \frac{2}{\alpha+m}.
    \]
    Using Hölder's inequality and the Gagliardo--Nirenberg inequality, we have 
    \begin{equation}
        \label{eq-lemma-3.1-mterm-case2}
        \begin{split}
        \norm{u_0}_{1}^{p\left(\frac{\alpha+m}{2}\right)} & = \norm{u(t)}_{1}^{p\left(\frac{\alpha+m}{2}\right)}  \, \leq \, \mathcal{L}^d\left( \mathrm{supp}\, u(t) \right)^{p\left(\frac{\alpha+m}{2}\right)\left(1-\frac{2}{p(\alpha+m)} \right)} \, \norm{u(t)}_{p\left(\frac{\alpha+m}{2}\right)}^{p\left(\frac{\alpha+m}{2}\right)} \\
        &\lesssim_{(m,\alpha,d,p)} \,  \mathcal{L}^d\left( \mathrm{supp}\, u(t) \right)^{p\left(\frac{\alpha+m}{2}\right)\left(1-\frac{2}{p(\alpha+m)} \right)} \, \norm{\nabla u^{\frac{\alpha+m}{2}}}_2^{p\vartheta} \, \norm{u_0}_{1}^{p(1-\vartheta)(\frac{\alpha+m}{2})},
        \end{split}
    \end{equation}
    where 
    \[
        \vartheta = \frac{\frac{\alpha+m}{2}-\frac{1}{p}}{\frac{1}{d}-\frac{1}{2}+\frac{\alpha+m}{2}} = \frac{d(\alpha+m)\left(1-\frac{2}{p(\alpha+m)} \right)}{2+d(\alpha+m-1)}.
    \]
    Note that in order to ensure that the conditions of the Gagliardo--Nirenberg inequality are met, $p$ must be less than the Sobolev exponent, i.e., 
    \[
    \frac{1}{p}>\frac{1}{2}-\frac{1}{d}.
    \]
     On the other hand, the original assumption on $p$ enforces 
     \[
        \frac{1}{p}< \frac{\alpha+m}{2}.
     \]
     Therefore, we must assure ourselves of the fact that 
     \[
        \frac{1}{2}-\frac{1}{d}  <  \frac{\alpha+m}{2}.
     \]
     For $d\in \left\{1,2\right\}$, this is guaranteed whenever $\alpha\in (-1,0)$ and $m>1$. When $d=3$, we require further that 
    \[
        \alpha > \frac{1}{3}-m.
    \]
    From \eqref{eq-lemma-3.1-mterm-case2}, we solve for the gradient term again to obtain
    \[
        \norm{\nabla u^{\frac{\alpha+m}{2}}}_2^{2} \, \gtrsim_{(m,\alpha,d,p)} \,  \norm{u_0}_{1}^{\alpha+m}\,\mathcal{L}^d\left( \mathrm{supp}\, u(t) \right)^{\frac{p}{\vartheta}\left(\alpha+m\right)\left(1-\frac{2}{p(\alpha+m)} \right)} = \norm{u_0}_{1}^{\alpha+m}\,\mathcal{L}^d\left( \mathrm{supp}\, u(t) \right)^{-\frac{d(\alpha+m-1)+2}{d}}.
    \]
    Integrating in time gives \eqref{eq-lemma-3.1-mterm-done} once again. Plugging the estimates \eqref{eq-lemma-3.1-nterm-done} and \eqref{eq-lemma-3.1-mterm-done} into \eqref{eq-lemma-3.1-initial} yields the desired inequality. 
    
\end{proof}

 We move on now to the weighted monotonicity formula. For this result, we rely heavily on the hard work already completed in \cite{FischerWTU,Fischer_Asymptotic,Fischer2016}, which will spare us much of the technicality of the argument, since we need only justify the addition of terms arising from the porous-medium equation.

\begin{lemma}
    \label{lemma-entropy-monotonicity}
    Let $u$ be a strongly energy dissipating solution satisfying \eqref{ass-u0}, \eqref{ass-nm}, and such that $u$ has finite speed of propagation. Let $x_0 \in \R^d\setminus \mathrm{supp} \, u_0$. There exists $\alpha \in \left(\max\{\frac{1}{2}-n,-m,-1\}, 2-n\right)$ and $\gamma \leq - d$ such that 
    \begin{equation}
        \label{eq-gamma-properties}
        \gamma + 4\left(\frac{\alpha+1}{n}\right)+d>0, \quad \gamma + 2\left(\frac{\alpha+1}{m-1} \right) +d >0,
    \end{equation}
    and we have for all $0\leq t_0 <t_1< T_w(x_0)$ (recall Definition \ref{defn-arrivaltime}):
    \begin{equation}
        \label{eq-entropy-monotonicity}
        \begin{split}
        \int_{\R^d} \abs{x-x_0}^\gamma u(t)^{\alpha+1} \, \bigg\vert_{t=t_0}^{t=t_1} \, & \geq \, \tau_1\int_{t_0}^{t_1}\left( \int_{\mathrm{supp}\, u(t)} \abs{x-x_0}^{\gamma+2\left( \frac{\alpha+1}{m-1} \right)} \right)^{-\frac{m-1}{\alpha+1}}\left( \int_{\R^d} \abs{x-x_0}^\gamma u(t)^{\alpha+1}\right)^{\frac{\alpha+m}{\alpha+1}} \, \mathrm{d}t \\
         & \hspace{1cm} + \tau_2\int_{t_0}^{t_1}\left( \int_{\mathrm{supp}\, u(t)} \abs{x-x_0}^{\gamma+4\left( \frac{\alpha+1}{n} \right)}  \right)^{-\frac{n}{\alpha+1}}\left(\int_{\R^d} \abs{x-x_0}^\gamma u(t)^{\alpha+1}\right)^{\frac{\alpha+n+1}{\alpha+1}} \, \mathrm{d}t.
        \end{split}
    \end{equation}
    Here, the constants $\tau_1= \tau_1(\alpha,m,\gamma,d)$ and $\tau_2 =\tau_2(\alpha,n,\gamma,d)$ are positive.
    
\end{lemma}

\begin{proof}
    Fix an arbitrary $x_0\in \R^d\setminus \mathrm{supp} \, u_0$. From here on, we shall suppress the dependence of $T_w(x_0)$ on $x_0$. Due to the property of finite speed of propagation, there exists a function $\psi\in C^\infty_c\left( (0,T_w)\times \R^d \right)$ such that 
    \[
        \psi(t,x) = \abs{x-x_0}^\gamma, \quad x\in \mathrm{supp}\, u(t), \quad t\in (0,T_w).
    \]
    Here we take $\gamma\leq -d$, but without further restrictions for now. Moreover, for any $\alpha\in (-1,0)$ and $\epsilon>0$, the function $\psi (u+\epsilon)^{\alpha}$ is a valid test function, and we have 
    \begin{equation}\label{eq-lemma-3.2-epsilon-regular}
        \frac{1}{\alpha+1}\int_{\R^d} \psi\left( u(t)+\epsilon\right)^{\alpha+1}\, \bigg\vert_{t_0}^{t_1} \, = \, \int_{t_0}^{t_1}\int_{\R^d} \nabla\left(\psi (u+\epsilon)^{\alpha}\right)\cdot u^n\nabla\Delta u \, -\, \int_{t_0}^{t_1}\int_{\R^d}\nabla\left(\psi (u+\epsilon)^\alpha \right)\cdot u^{m-1}\nabla u. 
    \end{equation}
    We turn first to the second integral on the right hand side: 
    \begin{equation}\label{eq-lemma-3.2-ibp}\begin{split}
        -\, \int_{t_0}^{t_1}\int_{\R^d}\nabla\left(\psi (u+\epsilon)^\alpha \right)\cdot u^{m-1}\nabla u \, & = \, -\alpha\int_{t_0}^{t_1}\int_{\R^d} \psi (u+\epsilon)^{\alpha-1}u^{m-1}\abs{\nabla u}^2 \\
        & \hspace{2cm}- \int_{t_0}^{t_1}\int_{\R^d} \nabla\psi\cdot (u+\epsilon)^{\alpha}u^{m-1}\nabla u \\
        & \geq \, - \int_{t_0}^{t_1}\int_{\R^d} \nabla\psi\cdot (u+\epsilon)^{\alpha}u^{m-1}\nabla u \\
        &= \, \int_{t_0}^{t_1}\int_{\R^d} \Delta \psi\, \left(\int_0^u (s+\epsilon)^{\alpha}s^{m-1} \, \mathrm{d}s\right).
        \end{split}
    \end{equation}
    Owing to restriction $\alpha>-m$, the following convergence holds point-wise almost everywhere:
    \[
        \int_0^u (s+\epsilon)^{\alpha}s^{m-1} \, \mathrm{d}s \longrightarrow \frac{u^{\alpha+m}}{\alpha+m}.
    \]
    Moreover, we have 
    \begin{align}
        \abs{\Delta \psi \, \left(\int_0^u (s+\epsilon)^{\alpha}s^{m-1} \, \mathrm{d}s\right)} = \abs{\Delta \psi \, \left(\int_0^u \left(\frac{s+\epsilon}{s}\right)^{\alpha}s^{\alpha+m-1} \, \mathrm{d}s\right)} \leq \frac{\abs{\Delta \psi} \, u^{\alpha+m}}{\alpha+m},
    \end{align}
    where the latter may be shown to be in $L^1\left((t_0,t_1)\times\R^d\right)$ via \eqref{eq-alphaentropy}. Indeed, whenever $$\alpha\in \left( \max\{\frac{1}{2}-n,-m,-1\},2-n \right),$$ we may apply the Gagliardo--Nirenberg inequality, together with Young's inequality to obtain the following: 
    \begin{align*}
        \int_{t_0}^{t_1} \norm{u(t)^{\alpha+m}}_{L^1} \, \mathrm dt \,& \lesssim_{(\alpha,n,m)} \, \int_{t_0}^{t_1}\norm{\nabla u(t)^{\frac{\alpha+m}{2}}}_{L^2}^{2}\, \mathrm d t \,  + \, (t_1-t_0)\, \left(\sup_{t_0\leq t\leq t_1}\,\norm{u(t)^{\alpha+1}}_{L^1}\right)^{\left(\frac{\alpha+m}{\alpha+1} \right)} \, < +\infty.
    \end{align*}
    The boundedness of the right side is a consequence of finite speed of propagation, which allows us to take $\zeta =1$ in \eqref{eq-alphaentropy}. Therefore, we may pass to the limit in \eqref{eq-lemma-3.2-ibp} using the Dominated Convergence Theorem and obtain: 
    \begin{equation}\label{eq-lemma-3.2-pass-to-limit1}
    \lim_{\epsilon\ra 0} \,\left( -\, \int_{t_0}^{t_1}\int_{\R^d}\nabla\left(\psi (u+\epsilon)^\alpha \right)\cdot u^{m-1}\nabla u \right) \, \geq \, \frac{1}{\alpha+m}\int_{t_0}^{t_1}\int_{\R^d} \Delta \psi \, u^{\alpha+m}.
    \end{equation}
    We compute 
    \[
        \mathbf{1}_{\{u>0\}}\Delta\psi = \mathbf{1}_{\{u>0\}} \cdot \gamma(\gamma-2+d)\abs{x-x_0}^{\gamma-2}.
    \]
    The coefficient in front is positive since $\gamma\leq -d$. We define now
    \[
        \tau_1 := \frac{\gamma(\gamma-2+d)}{\alpha+m}.
    \]
    Note that we have not yet made an explicit choice of $\alpha$, and the only condition which has been imposed is that $\alpha\in (\max\{\frac{1}{2}-n,-m,-1\},2-n)$ in order to use the estimate \eqref{eq-alphaentropy}. Turning now to the first integral on the right hand side of \eqref{eq-lemma-3.2-epsilon-regular}, we rely on the computations performed in \cite{FischerWTU,Fischer_Asymptotic,Fischer2016}. In \cite[Appendix]{FischerWTU}, it was proven that one may integrate by parts as one wishes and pass to the limit $\epsilon \ra 0$. The following consequence is important for our purposes: It is shown in \cite[Section 3.4]{Fischer_Asymptotic} that there exists a choice of $\alpha\in \left(\max\{\frac{1}{2}-n,-1\},0 \right)$, and a choice of $\gamma\leq -d$ such that 
    \begin{equation}\label{eq-aux-cond-n}
        \gamma+\left( \frac{\alpha+1}{n}\right) +d >0
    \end{equation}
    and 
    \begin{equation}\label{eq-fischers-hard-work}
        \lim_{\epsilon\ra 0} \, \int_{t_0}^{t_1}\int_{\R^d} \nabla\left(\psi (u+\epsilon)^{\alpha}\right)\cdot u^n\nabla\Delta u \, \geq \, \tau_2\int_{t_0}^{t_1}\int_{\R^d} \abs{x-x_0}^{\gamma-4}\, u^{\alpha+n+1}, 
    \end{equation}
    where $\tau_2>0$ depends on $(\alpha,\gamma,n,d)$\footnote{It should be noted that the two previously mentioned results hold originally for $n\in[2,32/11)$. However, due to a optimization of the procedure (s. \cite[Lemma 12]{Fischer2016}), this result was then extended to the full range $n\in [2,3)$ (s. \cite[Theorem 4]{Fischer2016})}. Our solution concept (recall Definition \ref{defn-solution}) entails all of the regularity to carry these results over into this setting, namely the regularity granted by the inequalities \eqref{eq-alphaentropy} and \eqref{eq-energyest1}. Furthermore, the assumption \eqref{ass-nm} tells us that if \eqref{eq-aux-cond-n} is fulfilled, then so too is
    \[
        \gamma + 2\left(\frac{\alpha+1}{m-1}\right)+d>0.
    \]
    Combining the estimates \eqref{eq-lemma-3.2-pass-to-limit1} and \eqref{eq-fischers-hard-work} in \eqref{eq-lemma-3.2-epsilon-regular} yields
    \begin{equation}
        \label{eq-lemma-3.2-prehol}
        \int_{\R^d} \abs{x-x_0}^{\gamma} \,u(t)^{\alpha+1}\, \bigg\vert_{t_0}^{t_1} \, \geq \, \tau_1\int_{t_0}^{t_1}\int_{\R^d} \abs{x-x_0}^{\gamma -2} \, u^{\alpha+m} \, + \, \tau_2\int_{t_0}^{t_1}\int_{\R^d} \abs{x-x_0}^{\gamma -4} \, u^{\alpha+n+1}.
    \end{equation}
    By Hölder's inequality, we have
    \begin{align*}
        \int_{\R^d} \abs{x-x_0}^{\gamma -2} \, u^{\alpha+m} \, \geq \, \left( \int_{\mathrm{supp}\, u(t)} \abs{x-x_0}^{\gamma+2\left( \frac{\alpha+1}{m-1}\right)}\right)^{-\frac{m-1}{\alpha+1}}\left(\int_{\R^d} \abs{x-x_0}^{\gamma}u^{\alpha+1} \right)^{\frac{\alpha+m}{\alpha+1}} \\
        \int_{\R^d} \abs{x-x_0}^{\gamma -4} \, u^{\alpha+n+1} \, \geq \, \left( \int_{\mathrm{supp}\, u(t)} \abs{x-x_0}^{\gamma+4\left( \frac{\alpha+1}{n}\right)}\right)^{-\frac{n}{\alpha+1}}\left(\int_{\R^d} \abs{x-x_0}^{\gamma}u^{\alpha+1} \right)^{\frac{\alpha+n+1}{\alpha+1}}.
    \end{align*}
    Thus, the result follows after substituting the above inequalities into \eqref{eq-lemma-3.2-prehol}. 

\end{proof}

\section{Proofs of the Main Results: The Lower Bound for $T_w(x_0)$}\label{section-lower}

This section is dedicated to the proofs of Theorem \ref{thm-fsop}, Theorem \ref{thm-main-result-lower}, and Corollary \ref{col-biggest-ball}. Throughout the proofs of this section, we will omit the dependence of $R_0(x_0)$, $T_r(x_0)$, $T_{R_0}(x_0)$, and $T_w(x_0)$ on $x_0$ when $x_0$ is apparent from context. The proofs employ a novel version of Stampacchia's lemma for a Stampacchia-type inequality with multiple weights. The statement and the proof is in the same vein as in \cite[Lemma 3.1]{GGDP}, which itself generalizes the original lemma from \cite{stamp63}. Note that in contrast to the Stampacchia-type lemmas of \cite[Lemma A.2]{GS} or \cite[Lemma B.4]{Taranets_Chur_Coating_Cylinder}, which apply to systems of functions which each fulfill a Stampacchia-type inequality, only \emph{one} function with multiple weights is present in this result.

\begin{lemma}
    \label{lemma-stampachhia-wait}
    Let $M:[0,b]\ra \R_{\geq 0}$ be a non-increasing function and let $N\in \N$. Suppose that there are constants $c_i>0$ $\alpha_i \geq  0$, $\beta_i>1$, $\kappa \geq 0$, and $\sigma\geq 0$ for $i\in \left\{1,\dots,N\right\}$ such that for any $s\geq 0$, $\delta>0$ satisfying $s+\delta\leq b$, we have 
    \begin{equation}
        \label{eq-stampacchia-form1}
        M(s+\delta) \leq \kappa(b-s)^\sigma + \sum_{i=1}^N \frac{c_i \cdot M(s)^{\beta_i}}{\delta^{\alpha_i}}.
    \end{equation}
    If for any $i=1,\dots, N$ we have $\sigma \geq \frac{\alpha_i}{\beta_i -1}$ (or simply $\sigma>0$ if all $\alpha_i=0$), and we also have  
    \begin{equation}
        \label{eq-stamp-condwait}        
        b^{\alpha_i} \geq N\cdot c_i \cdot 2^{\frac{\alpha_i\beta_i}{\beta_i-1}+\sigma}\cdot \left(1+2^\sigma\right)^{\beta_i} \cdot \left(\kappa b^\sigma +  M(0)\right)^{\beta_i-1}, \quad i=1,\dots, N,
    \end{equation}
    then 
    \[
         M(b) = 0.
    \]
    In the case $\kappa = 0$, the same result is true if
    \[
        b^{\alpha_i} \geq c_i \cdot 2^{\frac{\alpha_i\beta_i}{\beta_i-1}}\cdot M(0)^{\beta_i-1},\quad i=1,\dots, N.
    \]
\end{lemma}

\begin{proof}
    We will prove the case $\kappa>0$, following \cite[Lemma 3.1]{GGDP}. The proof of the case $\kappa=0$ is identical. Let $s_k = b-2^{-k}b$. Since $M$ is non-increasing, it suffices to show that 
    \[
        \lim_{k\ra \infty} M(s_k) = 0.
    \]
    We will prove by induction that 
    \[
        M(s_k) \leq \kappa b^\sigma2^{-\sigma(k-1)} + 2^{-\sigma k}\left(\kappa b^\sigma + M(0)\right), \quad k\in\N.
    \]
    For the base case, we have 
    \begin{align*}
        M(s_1) & \leq \kappa b^\sigma + \sum_{i=1}^N \frac{c_i \cdot M(0)^{\beta_i}}{(b/2)^{\alpha_i}} \\
        & \leq \kappa b^\sigma + \sum_{i=1}^N \frac{2^{-\sigma + \alpha_i - \frac{\alpha_i\beta_i}{\beta_i -1}} \cdot M(0)^{\beta_i}}{N\cdot (1+2^\sigma)^{\beta_i}\cdot \left(\kappa b^\sigma +  M(0)\right)^{\beta_i-1} } \\
        &\leq \kappa b^\sigma + 2^{-\sigma} \cdot \frac{M(0)}{N}\sum_{i=1}^N 2^{\frac{\alpha_i}{1-\beta_i}} \\
        &\leq  \kappa b^\sigma + 2^{-\sigma}M(0).
    \end{align*}
    Suppose now that for some $k\in \N$, we have 
    \[
        M(s_k) \leq \kappa b^\sigma2^{-\sigma(k-1)} + 2^{-\sigma k}\left(\kappa b^\sigma + M(0)\right).
    \]
    Then 
    \begin{align*}
        M(s_{k+1}) & \leq \kappa b^\sigma2^{-\sigma k} + \sum_{i=1}^N \frac{c_i \cdot M(s_k)^{\beta_i}}{(b/2^{k+1})^{\alpha_i}} \\
        & \leq \kappa b^\sigma2^{-\sigma k} + \sum_{i=1}^N \frac{2^{\alpha_i(k+1) - \frac{\alpha_i\beta_i}{\beta_i -1}-\sigma} \cdot \left(\kappa b^\sigma2^{-\sigma(k-1)} + 2^{-\sigma k}\left(\kappa b^\sigma + M(0)\right)\right)^{\beta_i}}{N\cdot \left(1+ 2^\sigma\right)^{\beta_i}\cdot \left(\kappa b^\sigma +  M(0)\right)^{\beta_i-1} } \\
        & \leq \kappa b^\sigma2^{-\sigma k} + \sum_{i=1}^N \frac{2^{\alpha_i(k+1) - \frac{\alpha_i\beta_i}{\beta_i -1} -\sigma\beta_i k-\sigma} \cdot \left(\kappa b^\sigma\left(1+ 2^{\sigma}\right) + M(0)\right)^{\beta_i}}{N\cdot \left(1+ 2^\sigma\right)^{\beta_i}\cdot \left(\kappa b^\sigma +  M(0)\right)^{\beta_i-1} } \\
        &\leq \kappa b^\sigma2^{-\sigma k} + 2^{-\sigma(k+1)} \cdot \frac{\left(\kappa b^\sigma + M(0)\right)}{N}\sum_{i=1}^N 2^{\alpha_i(k+1) - \frac{\alpha_i\beta_i}{\beta_i -1} -\sigma\beta_i k + \sigma(k+1)-\sigma}.
    \end{align*}
    The condition $\sigma \geq \frac{\alpha_i}{\beta_i-1}$ is enough to enforce that 
    \[
        \alpha_i(k+1) - \frac{\alpha_i\beta_i}{\beta_i -1} -\sigma\beta_i k + \sigma(k+1)-\sigma \leq 0
    \]
    for all $i=1,\dots,N$, whence 
    \[
      M(s_{k+1}) \leq  \kappa b^\sigma2^{-\sigma k} + 2^{-\sigma(k+1)} \left(\kappa b^\sigma + M(0)\right),
    \]
    and the proof is complete. 

\end{proof}

We are now prepared to prove Theorem \ref{thm-fsop}, Theorem \ref{thm-main-result-lower}, and Corollary \ref{col-biggest-ball}.
\begin{proof}[Proof of Theorem \eqref{thm-fsop}]
    Fix arbitrary $x_0 \in \R^d\setminus \mathrm{supp}\, u_0$ and $r\in (0,R_0)$. We must show that $T_r>0$. To that end, define 
    \[
    M:[T_{R_0},\infty)\times [0,R_0-r]\ra \R_{\geq0}, \quad (T,s)\mapsto \left( \sup_{T_{R_0}  \leq t\leq T} \, \int_{B_{R_0-s}} u(t) \right)^2.
    \]
    If we can find $T> T_{R_0}\geq 0$ such that $M(T,R_0-r) = 0$, then $0<T\leq T_r$, and the proof will be complete. 
    
    By Lemma \ref{lemma-local-L1}, we have for all $T>0$ and all $s,\delta$ such that $0\leq s <s+\delta \leq  R_0-r$:
    \[
        M(T,s+\delta) \, \lesssim_{(n,m,d)} \, \mathcal L^d( B_{R_0} )\, (T-T_{R_0})\left( \delta^{-(dm+2)} M(T,s)^{\frac{m+1}{2}} + \delta^{-(4+d(n+1))} M(T,s)^{\frac{n+2}{2}} \right).
    \]
    We have used here $B_{R_0} \cap \mathrm{supp} \ u(t) = \nothing$ for any $t<T_{R_0}$, which allows one to neglect the initial terms which appear on the right hand side of \eqref{eq-local-L1}. Hence all of the conditions of Lemma \ref{lemma-stampachhia-wait} are fulfilled with $\kappa =0$. Therefore, for any $T>T_{R_0}$ such that \emph{both of the following are fulfilled}:
    \begin{align*}
        (R_0-r)^{dm+2} \, \gtrsim_{(n,m,d)} \, \mathcal L^d( B_{R_0} )\, (T-T_{R_0}) \left( \sup_{T_{R_0}  \leq t\leq T} \, \int_{B_{R_0-s}} u(t) \right)^{m-1} \\
        (R_0-r)^{d(n+1)+4} \, \gtrsim_{(n,m,d)} \, \mathcal L^d( B_{R_0} )\, (T-T_{R_0}) \left( \sup_{T_{R_0}  \leq t\leq T} \, \int_{B_{R_0-s}} u(t) \right)^n,
    \end{align*}
    we have $M(T,R_0-r) = 0$. Using the fact that $\mathcal L^d(B_{R_0}) \, \lesssim_d  \, R_0^d $ and that 
    \[
        \sup_{T_{R_0}  \leq t\leq T} \, \int_{B_{R_0-s}} u(t) \leq \, \norm{u_0}_{L^\infty_t L^1_x},
    \]
    we find that the above conditions are fulfilled if $T$ satisfies
    \begin{equation}
        \label{eq-fsop-time1}
        T\leq T_{R_0} + c_1\min\left\{ \, (R_0-r)^{dm+2}R_0^{-d} \norm{u}_{L^\infty_tL^1_x}^{1-m} \, , \, (R_0-r)^{d(n+1)+4}R_0^{-d} \norm{u}_{L^\infty_tL^1_x}^{-n} \right\}.
    \end{equation}
    Here the constant $c_1>0$ is the implicit constant which has been carried through the computations until this point. It depends only on $(n,m,d)$. For any such $T$, we have $T_r\geq T$, which yields the bound 
    \begin{equation}
        \label{eq-fsop-time2}
        T_r \, \geq  \, T_{R_0} + c_1\min\left\{ \, (R_0-r)^{dm+2}R_0^{-d}\norm{u}_{L^\infty_tL^1_x}^{1-m} \, , \, (R_0-r)^{d(n+1)+4}R_0^{-d} \norm{u}_{L^\infty_tL^1_x}^{-n} \right\} \, > \, 0.
    \end{equation}
    
\end{proof}

\begin{remark}
    Now that finite speed of propagation is established, we know that $\norm{u}_{L^\infty_t L^1_x} = \norm{u_0}_{L^1_x}$.
\end{remark}

\begin{proof}[Proof of Theorem \ref{thm-main-result-lower}] 
For the first part, the bound \eqref{eq-fsop-time2} implies immediately that 
\[
    \lim_{r\ra 0} \, \left( T_{R_0} + c_1\min\left\{ \, (R_0-r)^{dm+2}R_0^{-d} \norm{u_0}_{L^1}^{1-m} \, , \, (R_0-r)^{d(n+1)+4}R_0^{-d} \norm{u_0}_{L^1}^{-n} \right\} \right) \, \leq \, \lim_{r\ra 0}\, T_r \, = \, T_w.
\]
Computing the limit on the left side yields precisely the desired inequality \eqref{eq-spreadingrate-lower}. For the second part of the theorem, suppose that there exists $\kappa = \kappa(x_0) \in (0,\infty) $ and $r_0 = r_0(x_0) \in (0,\infty)$ such that the flatness condition \eqref{eq-wtp-flatness-condition} holds. By Poincar\'{e}'s inequality, we have
\begin{equation}\label{eq-intial-decay}
        \sup_{r\leq r_0} \, \frac{1}{r^{d+\sigma}} \, \int_{B_{R_0+r}\setminus B_r} u_0^2 \, \,\lesssim \, \,\sup_{r\leq r_0} \, \frac{r^2}{r^{d+\sigma}} \,\int_{B_{R_0+r}\setminus B_r} \abs{\nabla u_0}^2 \,  \leq \, \kappa^2.
\end{equation}

\noindent We define the function $W:(0,\infty)\times [0,r_0] \, \ra \, \R_{\geq0}$ by 
\[
    W(T,s) \,:= \, \sup_{t\leq T} \, \int_{B_{R_0+r_0-s}} u(t)^2.
\]
If we can find $T>0$ such that $W(T,r_0) = 0$, then $0<T\leq T_{R_0}$. To this end, we also proceed with a Stampacchia argument, but using Lemma \ref{lemma-local-L2} instead. Applying the inequality \eqref{eq-local-L2} yields the following estimate for any $T>0$ and every $0\leq s<s+\delta \leq r_0$:
\begin{equation}
    \label{eq-wtp-stampacchia-form1} \begin{split}
    W(T,s+\delta) \, & \lesssim_{(n,m,d)} \, \int_{B_{R_0+r_0-s}} u_0^2 \,  + \, \delta^2 \int_{B_{R_0+r_0-s}}\abs{\nabla u_0}^2 \\ 
    & \hspace{1cm} + T\left(\delta^{-2-\frac{d(m-1)}{2}} W(T,s)^{\frac{m+1}{2}} + \,\delta^{-4-\frac{dn}{2}}W(T,s)^{\frac{n+2}{2}} \right).
    \end{split}
\end{equation}
Now, we estimate the right hand side above by applying \eqref{eq-intial-decay} and using the fact that $\delta\leq r_0-s$:
\begin{equation}
    \label{eq-wtp-stampacchia-form2}
     \begin{split}
    W(T,s+\delta) \, \lesssim_{(n,m,d)} \, \kappa^2(r_0-s)^{d+\sigma} \, + \, T\left(\delta^{-2-\frac{d(m-1)}{2}} W(T,s)^{\frac{m+1}{2}} + \,\delta^{-4-\frac{dn}{2}}W(T,s)^{\frac{n+2}{2}} \right).
    \end{split}
\end{equation}
We have once again arrived at the required condition \eqref{eq-stampacchia-form1} of Lemma \ref{lemma-stampachhia-wait} where the parameters $\alpha_i,\beta_i$ are given by 
\[
    \alpha_i = \, 4+\frac{dn}{2} \, , \,2+\frac{d(m-1)}{2}
\]
and
\[
    \beta_i = \, \frac{n+2}{2} \, , \, \frac{m+1}{2}.
\]
Recall from the statement of Theorem \ref{thm-main-result-lower} that we have defined a parameter $\sigma$ as 
\[
    \sigma = \max\left\{\frac{8}{n}, \frac{4}{m-1}\right\}
.\]
For the application of Lemma \ref{lemma-stampachhia-wait}, it is important to compute the ratios $\alpha_i/(\beta_i -1)$, as the maximum of these ratios must be less than or equal to $d+\sigma$ in order to apply Lemma \ref{lemma-stampachhia-wait}. We see that
\[
    \max_i\left\{ \frac{\alpha_i}{\beta_i-1}\right\} = d+\sigma.
\]
Thus, the conditions to apply Lemma \ref{lemma-stampachhia-wait} are met. The following is therefore true: For any $T>0$ \emph{which satisfies both of the following conditions}:
\begin{equation}\label{eq-wtp-cond1}
    r_0^{\alpha_i} \, \gtrsim_{(n,m,d)} \, T\left( \kappa^2r_0^{d+\sigma} + W(T,0)\right)^{\beta_i-1}, \qquad i=1,2,
\end{equation}
we have $W(T,r_0)=0$. In view of the estimate 
\[
     W(T,0) = \sup_{t\leq T} \, \int_{B_{R_0+r_0}} u(t)^2 \,\leq\, \norm{u}_{L^\infty_tL^2_x}^2 \, \lesssim_{(n,m,d)} \norm{\nabla u_0}_{L^2}^{\frac{2d}{d+2}}\norm{u_0}_{L^1}^{\frac{4}{d+2}}  <+\infty,
\]
the conditions on $T$ hold if 
\[
T \, \lesssim_{(n,m,d)} \, \min\left\{ r_0^{2+\frac{d(m-1)}{2}}\left( \kappa^2r_0^{d+\sigma} + \norm{u}_{L^\infty_t L^2_x}^2\right)^{\frac{1-m}{2}} \, , \, r_0^{4+\frac{dn}{2}}\left( \kappa^2r_0^{d+\sigma} + \norm{u}_{L^\infty_t L^2_x}^2\right)^{-\frac{n}{2}} \right\}.
\]
From this, the desired estimate \eqref{eq-wtp-duration} follows after defining the implicit constant to be $c_2$. 

\end{proof}

\begin{proof}[Proof of Corollary \ref{col-biggest-ball}]
      Begin by defining the set 
    \begin{equation}
        S_t := \left\{x\in \R^d\setminus \mathrm{supp} \, u_0, \quad t\geq T_w(x) \right\}.
    \end{equation}
    Since $S_t$ gives the maximum expansion of the support until time $t$, we must have 
    \[
        \mathrm{supp} \, u(t) \setminus \mathrm{supp} \, u_0\subset \, S_t. 
    \]
    Now using the lower bound \eqref{eq-spreadingrate-lower} from Theorem \ref{thm-main-result-lower}, $S_t$ satisfies: 
    \begin{equation}\label{eq-St-representation}\begin{split}
        S_t & := \left\{x\in \R^d\setminus \mathrm{supp} \, u_0, \quad t\geq T_w(x) \right\} \\
        & \subset \left\{ x\in \R^d\setminus \mathrm{supp} \, u_0, \quad t\geq\, c_1\min\left\{ \norm{u_0}_{L^1}^{1-m} R_0(x)^{d(m-1)+2}, \norm{u_0}_{L^1}^{-n} R_0(x)^{dn+4} \right\} \right\} \\
        & \subset \left\{ x\in \R^d\setminus \mathrm{supp} \, u_0, \quad R_0(x)\,\leq\, \max\left\{\left( c_1\inv \norm{u_0}_{L^1}^{m-1} t\right)^{\frac{1}{d(m-1)+2}} \, , \, \left( c_1\inv \norm{u_0}_{L^1}^{n} t\right)^{\frac{1}{dn+4}}\right\} \right\},
        \end{split}
    \end{equation}
    For the given point $x_0\in \R^d$, we may start by including the support into the following ball: 
    \[
        \mathrm{supp} \ u(t) \subset \, B\left(x_0, \sup_{y\in \mathrm{supp} \, u(t)} \, \mathrm{dist}\left(x_0,y \right) \right).
    \]
    We now look to give an explicit estimate of the radius of this ball. We first look towards $y\in \mathrm{supp}\, u_0$, for which we simply have 
    \[
        \mathrm{dist}(x_0,y) \leq R_0(x_0) + \mathrm{diam} (\mathrm{supp}\,u_0).
    \]
    If $y\in \mathrm{supp}\, u(t) \setminus \mathrm{supp}\, u_0 \subset S_t$, then there exists a point $y^*\in \partial \left(\mathrm{supp}\, u_0 \right)$ such that 
    \[
        \mathrm{dist}(y,y^*) = R_0(y).
    \]
    Thus, applying the triangle inequality and \eqref{eq-St-representation}, we obtain:
    \begin{equation} \label{eq-max-radius}
    \begin{split}
                \mathrm{dist}(x_0,y) & \leq R_0(x_0)+ \mathrm{diam} (\mathrm{supp}\,u_0) \\
                & \hspace{1cm}+ \max\left\{\left( c_1\inv \norm{u_0}_{L^1}^{m-1} t\right)^{\frac{1}{d(m-1)+2}} \, , \, \left( c_1\inv \norm{u_0}_{L^1}^{n} t\right)^{\frac{1}{dn+4}}\right\} =: R(x_0;t) .
                \end{split}
    \end{equation}
    This establishes the uniform estimate 
    \[
        \sup_{y\in \mathrm{supp} \, u(t)} \, \mathrm{dist}\left(x_0,y \right) \, \leq \, R(x_0;t).
    \]
    To obtain the expression for $R(x_0;t)$ found in \eqref{eq-biggest-ball-radius}, we begin by seeing that for all $n,m$ satisfying \eqref{ass-nm}, we have 
    \[
        \frac{1}{dn+4} <\frac{1}{d(m-1)+2} <1.
    \]
    Now looking to the maximum in \eqref{eq-max-radius}, the previous consideration shows that the maximum is obtained by the term with exponent $1/(dn+4)$ for small times. The functions in the maximum have one point of intersection: In this case it occurs at 
    \[
        t^* :=  c_1\norm{u_0}_{L^1}^{\frac{2(n+2-2m)}{d(n+1-m)+2}}.
    \]
    For all $t>t^*$, it must be the case that the maximum is attained by the term with exponent $1/(d(m-1)+2)$. This yields the form seen in \eqref{eq-biggest-ball-radius}. 
    
\end{proof}

\section{Proofs of the Main Results: The Upper Bound for $T_w(x_0)$}\label{section-upper}

In this section, we will combine all of the previous results from Theorem \ref{thm-main-result-lower}, Corollary \ref{col-biggest-ball}, Lemma \ref{lemma-entropy-lower}, and Lemma \ref{lemma-entropy-monotonicity} in order to prove Theorem \ref{thm-main-result-upper}. 

\begin{proof}[Proof of Theorem \ref{thm-main-result-upper}]
    This proof closely follows the the approach of \cite{Fischer_Asymptotic} with some additional arguments required to properly treat the individual scaling behaviors. 
    Let $x_0\in \R^d\setminus \mathrm{supp}\, u_0$ be given. We omit all dependencies on $x_0$, since no other points will be considered. For the rest of this proof, we fix
    \begin{equation}
        \label{eq-def-t0}
        t_0 = c_1 \min\left\{  c_1\norm{u_0}_{L^1}^{-n}\left(R_0+\mathrm{diam}(\mathrm{supp}\, u_0)\right)^{dn+4}\, ,\, \norm{u_0}_{L^1}^{1-m} \left(R_0+\mathrm{diam}(\mathrm{supp}\, u_0)\right)^{d(m-1)+2} \right\}.
    \end{equation} 
    The goal is to show that $T_w \, \leq \, c_3t_0$ for some constant $c_3>0$ which does not depend on $x_0$ or the mass of $u$.
    By Lemma \ref{lemma-entropy-monotonicity} we may find $\alpha \in (-1,0)$, $\gamma <\alpha d$, and $\tau_1,\tau_2 >0$ such that 
    \[
    \gamma+4\left(\frac{\alpha+1}{n}\right)+d>0,\quad \gamma+2\left(\frac{\alpha+1}{m-1}\right)+d>0,
    \]
    and the following holds for all $0\leq t_1 < t_2<T_w$:
    \begin{equation}
        \label{eq-monotonicity-app-1}
         \begin{split}
        \int_{\R^d} \abs{x-x_0}^\gamma u(t)^{\alpha+1} \, \bigg\vert_{t=t_1}^{t=t_2} \, & \geq \, \tau_1\int_{t_1}^{t_2}\left( \int_{\mathrm{supp}\, u(t)} \abs{x-x_0}^{\gamma+2\left( \frac{\alpha+1}{m-1} \right)} \right)^{-\frac{m-1}{\alpha+1}}\left( \int_{\R^d} \abs{x-x_0}^\gamma u(t)^{\alpha+1}\right)^{\frac{\alpha+m}{\alpha+1}} \, \mathrm{d}t \\
         & \hspace{0.5cm} + \tau_2\int_{t_1}^{t_2}\left( \int_{\mathrm{supp}\, u(t)} \abs{x-x_0}^{\gamma+4\left( \frac{\alpha+1}{n} \right)}  \right)^{-\frac{n}{\alpha+1}}\left(\int_{\R^d} \abs{x-x_0}^\gamma u(t)^{\alpha+1}\right)^{\frac{\alpha+n+1}{\alpha+1}} \, \mathrm{d}t.
        \end{split}
    \end{equation}
    While there is generally no closed form solution to the ODE 
    \[
    y'(t) = q_1(t)y^{p_1} + q_2(t)y^{p_2}
    \]
    for which we may use the comparison principle, we can use the comparison principle individually with each term on the right hand side of \eqref{eq-monotonicity-app-1}; i.e., we may consider 
    \[
     y'(t) \, \geq q_1(t) y(t)^{p_1} \quad \text{ and } \quad y'(t) \, \geq \, q_2(t)y(t)^{p_2}
    \]
    individually and apply comparison principles. This implies that for any $0\leq t_1<t_2<T_w$, the following holds:  
    \begin{equation}
        \label{eq-comparison-principle} \begin{split}
        & \int_{\R^d} \abs{x-x_0}^\gamma u(t_2)^{\alpha+1} \\  &\geq \,  \max\left\{ \left[ \left( \int_{\R^d} \abs{x-x_0}^\gamma u(t_1)^{\alpha+1} \right)^{-\frac{m-1}{\alpha+1}} - \tilde\tau_1 \int_{t_1}^{t_2}\left( \int_{\mathrm{supp}\, u(t)} \abs{x-x_0}^{\gamma+2\left( \frac{\alpha+1}{m-1} \right)}  \right)^{-\frac{m-1}{\alpha+1}}\, \mathrm{d}t\right]^{-\frac{\alpha+1}{m-1}} ,   \right. \\
         & \hspace{2cm}\left. \left[ \left( \int_{\R^d} \abs{x-x_0}^\gamma u(t_1)^{\alpha+1} \right)^{-\frac{n}{\alpha+1}} - \tilde\tau_2 \int_{t_1}^{t_2}\left( \int_{\mathrm{supp}\, u(t)} \abs{x-x_0}^{\gamma+4\left( \frac{\alpha+1}{n} \right)}  \right)^{-\frac{n}{\alpha+1}}\, \mathrm{d}t\right]^{-\frac{\alpha+1}{n}} \right\}.
        \end{split}
    \end{equation}
    Due to $t_2<T_w$ and Theorem \ref{thm-fsop} (which ensures that $T_w>0$), the left hand side of \eqref{eq-comparison-principle} is finite. Therefore, it must be the case that 
    \begin{equation}
        \label{eq-finiteness-necessary-cond}
        \begin{split}
            \left(\int_{\R^d} \abs{x-x_0}^\gamma u(t_1)^{\alpha+1} \right)^{-\frac{m-1}{\alpha+1}} \,\gtrsim_{(\alpha,\gamma,m,d)} \,  \int_{t_1}^{t_2}\left( \int_{\mathrm{supp}\, u(t)} \abs{x-x_0}^{\gamma+2\left( \frac{\alpha+1}{m-1} \right)}\right)^{-\frac{m-1}{\alpha+1}}\, \mathrm{d}t  \\
            \left( \int_{\R^d} \abs{x-x_0}^\gamma u(t_1)^{\alpha+1} \right)^{-\frac{n}{\alpha+1}} \, \gtrsim_{(\alpha,\gamma,n,d)}\, \int_{t_1}^{t_2}\left( \int_{\mathrm{supp}\, u(t)} \abs{x-x_0}^{\gamma+4\left( \frac{\alpha+1}{n} \right)}  \right)^{-\frac{n}{\alpha+1}}\, \mathrm{d}t.
        \end{split}
    \end{equation}
    At this point, we must distinguish two cases: either $t_0\leq t^*$ or $t_0> t^*$ (recall the Definition of $t^*$ from Corollary \ref{col-biggest-ball}). We consider first $t^*\leq t_0$. In this case, we consider the first inequality in \eqref{eq-finiteness-necessary-cond} with $t_1 = t_0$. We shall compute the integral on the right and apply Lemma \ref{lemma-entropy-lower} to the integral on the left. To begin, we note that for all $t\geq t_0>t^*$:
    \begin{equation}
        \label{eq-t0-geq-t*-R(t)}
        R(t) \, \leq \, 2\left(c_1\inv \norm{u_0}_{L^1}^{m-1}t \right)^{\frac{1}{d(m-1)+2}}.  
    \end{equation}
    Applying \eqref{eq-t0-geq-t*-R(t)} yields for each $t_2\in (t_0,T_w)$:
    \begin{equation}
        \label{eq-fc-1-rhs}\begin{split}
        &\hspace{1cm}\int_{t_0}^{t_2}\left( \int_{\mathrm{supp}\, u(t)} \abs{x-x_0}^{\gamma+2\left( \frac{\alpha+1}{m-1} \right)}\right)^{-\frac{m-1}{\alpha+1}} \, \, \mathrm{d}t \, \geq \, \int_{t_0}^{t_2}\left( \int_{B(x_0, R(t))} \abs{x-x_0}^{\gamma+2\left( \frac{\alpha+1}{m-1} \right)}\right)^{-\frac{m-1}{\alpha+1}} \, \mathrm{d}t \\
        & \hspace{6cm} = \left( \gamma+2\left( \frac{\alpha+1}{m-1} \right)+d\right)\inv  \int_{t_0}^{t_2} R(t)^{-\frac{m-1}{\alpha+1}\left(\gamma+2\left( \frac{\alpha+1}{m-1} \right)+d\right)} \, \mathrm{d}t \\
        & \hspace{6cm} = C(\alpha,\gamma,m,n,d) \, \norm{u_0}_{L^1}^{-\frac{(m-1)^2}{(d(m-1)+2)(\alpha+1)}\left(\gamma+2\left( \frac{\alpha+1}{m-1} \right)+d\right)} \\
        &\hspace{7.5cm} \times \int_{t_0}^{t_2} t^{-\frac{m-1}{(d(m-1)+2)(\alpha+1)}\left(\gamma+2\left( \frac{\alpha+1}{m-1} \right)+d\right)} \\
        &\hspace{6cm} = C(\alpha,\gamma,m,n,d) \, \norm{u_0}_{L^1}^{-\frac{(m-1)^2}{(d(m-1)+2)(\alpha+1)}\left(\gamma+2\left( \frac{\alpha+1}{m-1} \right)+d\right)} \\
        &\hspace{9.5cm} \times t^{\frac{(\alpha d-\gamma)(m-1)}{(d(m-1)+2)(\alpha+1)}} \, \bigg\vert_{t_0}^{t_2}.
        \end{split}
    \end{equation}
    Note that the constant $C(\alpha,\gamma,m,d)$ above is positive due to $\gamma + 2\left(\frac{\alpha+1}{m-1}\right)+d>0$, and the final power of $t$ in the last line is positive due to $\gamma\leq -d<\alpha d$. Now for the left hand side of the first inequality in \eqref{eq-finiteness-necessary-cond}, we apply that $\gamma<0$ along with Lemma \ref{lemma-entropy-lower} to obtain 
    \begin{equation}
        \label{eq-fc-1-lhs}\begin{split}
         &\left(\int_{\R^d} \abs{x-x_0}^\gamma u(t_0)^{\alpha+1} \right)^{-\frac{m-1}{\alpha+1}} \,  \leq \, R(t_0)^{-\frac{\gamma(m-1)}{\alpha+1}}\left(\int_{\R^d}  u(t_0)^{\alpha+1} \right)^{-\frac{m-1}{\alpha+1}} \\
         & \hspace{1.5cm}\leq C(\alpha,m,n,d) \, R(t_0)^{-\frac{\gamma(m-1)}{\alpha+1}}\left(  \norm{u_0}_{L^1}^{\alpha+m} \int_0^{t_0} \mathcal{L}^d\left( \mathrm{supp}\, u(t) \right)^{-\frac{d(\alpha+m-1)+2}{d}} \, \mathrm{d}t\right)^{-\frac{m-1}{\alpha+1}} \\
         &\hspace{1.5cm} = C(\alpha,m,n,d)  R(t_0)^{-\frac{\gamma(m-1)}{\alpha+1}}\left(  \norm{u_0}_{L^1}^{\alpha+m} \int_0^{t_0}R(t)^{-(d(\alpha+m-1)+2)} \, \mathrm{d}t\right)^{-\frac{m-1}{\alpha+1}} \\ 
         &\hspace{1.5cm} \leq  C(\alpha,m,n,d)  R(t_0)^{-\frac{\gamma(m-1)}{\alpha+1}}\left(  \norm{u_0}_{L^1}^{\alpha+m} \, t_0R(t_0)^{-(d(\alpha+m-1)+2)} \, \mathrm{d}t\right)^{-\frac{m-1}{\alpha+1}} \\
         &\hspace{1.5cm} \leq  C(\alpha,m,n,d) \norm{u_0}_{L^1}^{-\frac{(m-1)(\alpha+m)}{\alpha+1}-\frac{(\gamma+d(\alpha+m-1)+2)(m-1)^2}{(\alpha+1)(d(m-1)+2)}} \, t_0^{1-\frac{(\gamma+d(\alpha+m-1)+2)(m-1)}{(\alpha+1)(d(m-1)+2)}} \\
         &\hspace{1.5cm} =  C(\alpha,m,n,d) \norm{u_0}_{L^1}^{-\frac{(m-1)(\alpha+m)}{\alpha+1}-\frac{(\gamma+d(\alpha+m-1)+2)(m-1)^2}{(\alpha+1)(d(m-1)+2)}} \, t_0^{\frac{(\alpha d-\gamma)(m-1)}{(d(m-1)+2)(\alpha+1)}}.
        \end{split}
    \end{equation}
    Note that we have again used \eqref{eq-t0-geq-t*-R(t)} to estimate $R(t_0)$, and we used the fact that $d(\alpha+m-1)+2>0$ in the third inequality from the bottom. In order to ensure this in dimension three, we require $\alpha > 1/3 -m$. When $m>4/3$, we are permitted to use any $\alpha\in (-1,0)$ which is compatible with the other conditions. Plugging \eqref{eq-fc-1-rhs} and \eqref{eq-fc-1-lhs} into the first inequality of \eqref{eq-finiteness-necessary-cond}, there exists a constant depending only on $(\alpha,\gamma,m,n,d)$ such that for any $t_2\in (t_0,T_w)$:
    \begin{equation}
        \label{eq-case-t0-geq-t*-done}
        \begin{split}
         &\norm{u_0}_{L^1}^{-\frac{(m-1)(\alpha+m)}{\alpha+1}-\frac{(\gamma+d(\alpha+m-1)+2)(m-1)^2}{(\alpha+1)(d(m-1)+2)}} \, t_1^{\frac{(\alpha d-\gamma)(m-1)}{(d(m-1)+2)(\alpha+1)}} \\
         & \hspace{2cm}\gtrsim_{(\alpha,\gamma,m,n,d)} \,\norm{u_0}_{L^1}^{-\frac{(m-1)^2}{(d(m-1)+2)(\alpha+1)}\left(\gamma+2\left( \frac{\alpha+1}{m-1} \right)+d\right)} t^{\frac{(\alpha d-\gamma)(m-1)}{(d(m-1)+2)(\alpha+1)}} \, \bigg\vert_{t_1}^{t_2}.
         \end{split}
    \end{equation}
    We note that 
    \[
     -\frac{(m-1)(\alpha+m)}{\alpha+1}-\frac{(\gamma+d(\alpha+m-1)+2)(m-1)^2}{(\alpha+1)(d(m-1)+2)} =-\frac{(m-1)^2}{(d(m-1)+2)(\alpha+1)}\left(\gamma+2\left( \frac{\alpha+1}{m-1} \right)+d\right),
    \]
    hence the $\norm{u_0}_{L^1}$-terms in \eqref{eq-case-t0-geq-t*-done} cancel, leaving us with 
    \begin{equation} \label{eq-m-dominates}
        t_2 \, \lesssim_{(\alpha,\gamma,n,m,d)} \, t_0.
    \end{equation}
    Since this estimate is uniform for $t_2\in (t_0,T_w)$, $T_w$ satisfies it as well. 

    We now consider the case that $t_0 < t^*$. In this eventuality, there are two sub-cases: Either $T_w$ passes before the scaling behavior switches ($T_w<t^*$), or $T_w$ comes afterwards ($T_w>t^*$). The latter case is the more-complicated of the two. We shall now show, however, that each case leads to the same estimate (with different constants). First we address the case $t_0<T_w <t^*$, in which the proof is the same as in \cite{Fischer_Asymptotic}. We first have for all $t\in [t_0,T_w)$:
    \begin{equation}
         \label{eq-t0-leq-t*-R(t)}
        R(t) \, \leq \, 2\left(c_1\inv \norm{u_0}_{L^1}^{n}t \right)^{\frac{1}{dn+4}}.
    \end{equation}
    Using \eqref{eq-t0-leq-t*-R(t)}, we further estimate the \emph{second} inequality of \eqref{eq-finiteness-necessary-cond} in the same way as before (the computation is identical up to replacing $m-1$ with $n$ and $2$ with $4$ where one see it). Applying \eqref{eq-t0-leq-t*-R(t)} yields for each $t_2\in (t_0,T_w)$:
    \begin{equation}
        \label{eq-fc-2-rhs}\begin{split}
        &\int_{t_0}^{t_2}\left( \int_{\mathrm{supp}\, u(t)} \abs{x-x_0}^{\gamma+4\left( \frac{\alpha+1}{n} \right)}\right)^{-\frac{n}{\alpha+1}} \, \, \mathrm{d}t \, \geq \, \int_{t_0}^{t_2}\left( \int_{B(x_0, R(t))} \abs{x-x_0}^{\gamma+4\left( \frac{\alpha+1}{n} \right)}\right)^{-\frac{n}{\alpha+1}} \, \mathrm{d}t \\
        & \hspace{6cm} = \left( \gamma+4\left( \frac{\alpha+1}{n} \right)+d\right)\inv  \int_{t_0}^{t_2} R(t)^{-\frac{n}{\alpha+1}\left(\gamma+4\left( \frac{\alpha+1}{n} \right)+d\right)} \, \mathrm{d}t \\
        & \hspace{6cm} = C(\alpha,\gamma,n,m,d) \, \norm{u_0}_{L^1}^{-\frac{n^2}{(dn+4)(\alpha+1)}\left(\gamma+4\left( \frac{\alpha+1}{n} \right)+d\right)} \\
        &\hspace{7.5cm} \times \int_{t_0}^{t_2} t^{-\frac{n}{(dn+4)(\alpha+1)}\left(\gamma+4\left( \frac{\alpha+1}{n} \right)+d\right)} \\
        &\hspace{6cm} = C(\alpha,\gamma,n,m,d) \, \norm{u_0}_{L^1}^{-\frac{n^2}{(dn+4)(\alpha+1)}\left(\gamma+4\left( \frac{\alpha+1}{n} \right)+d\right)} \\
        &\hspace{9.5cm} \times t^{\frac{(\alpha d-\gamma)n}{(dn+4)(\alpha+1)}} \, \bigg\vert_{t_0}^{t_2}.
        \end{split}
    \end{equation}
    Note that in this case, the constant $C(\alpha,\gamma,n,m,d)$ above is positive due to $\gamma + 4\left(\frac{\alpha+1}{n}\right)+d>0$, and the final power of $t$ in the last line is positive due to $\gamma\leq -d<\alpha d$. Turning to the left hand side of the second inequality in \eqref{eq-finiteness-necessary-cond}, we use the fact that $\gamma<0$ along with Lemma \ref{lemma-entropy-lower} to obtain:
    \begin{equation}
        \label{eq-fc-2-lhs}\begin{split}
         &\left(\int_{\R^d} \abs{x-x_0}^\gamma u(t_0)^{\alpha+1} \right)^{-\frac{n}{\alpha+1}} \,  \leq \, R(t_0)^{-\frac{\gamma n}{\alpha+1}}\left(\int_{\R^d}  u(t_0)^{\alpha+1} \right)^{-\frac{n}{\alpha+1}} \\
         & \hspace{1.5cm}  \leq C(\alpha,n,m,d) R(t_0)^{-\frac{\gamma n}{\alpha+1}}\left(  \norm{u_0}_{L^1}^{\alpha+n+1} \int_0^{t_0} \mathcal{L}^d\left( \mathrm{supp}\, u(t) \right)^{-\frac{d(\alpha+n)+4}{d}} \, \mathrm{d}t\right)^{-\frac{n}{\alpha+1}} \\
         &\hspace{1.5cm} = C(\alpha,n,m,d)  R(t_0)^{-\frac{\gamma n}{\alpha+1}}\left(  \norm{u_0}_{L^1}^{\alpha+n+1} \int_0^{t_0}R(t)^{-(d(\alpha+n)+4)} \, \mathrm{d}t\right)^{-\frac{n}{\alpha+1}} \\ 
         & \hspace{1.5cm}\leq C(\alpha,n,m,d)  R(t_0)^{-\frac{\gamma n}{\alpha+1}}\left(  \norm{u_0}_{L^1}^{\alpha+n+1} \, t_0R(t_0)^{-(d(\alpha+n)+4)} \, \mathrm{d}t\right)^{-\frac{n}{\alpha+1}} \\
         &\hspace{1.5cm} \leq C(\alpha,n,m,d) \norm{u_0}_{L^1}^{-\frac{n(\alpha+n+1)}{\alpha+1}-\frac{(\gamma+d(\alpha+n)+4)n^2}{(\alpha+1)(dn+4)}} \, t_0^{1-\frac{(\gamma+d(\alpha+n)+4)(m-1)}{(\alpha+1)(dn+4)}} \\
         & \hspace{1.5cm}= C(\alpha,n,m,d) \norm{u_0}_{L^1}^{-\frac{n(\alpha+n+1)}{\alpha+1}-\frac{(\gamma+d(\alpha+n)+4)n^2}{(\alpha+1)(dn+4)}} \, t_0^{\frac{(\alpha d-\gamma)n}{(dn+4)(\alpha+1)}}.
        \end{split}
    \end{equation}
    Note that we have once again used \eqref{eq-t0-leq-t*-R(t)} to estimate $R(t_1)$. Plugging \eqref{eq-fc-2-rhs} and \eqref{eq-fc-2-lhs} into the second inequality of \eqref{eq-finiteness-necessary-cond}, there exists a constant depending only on $(\alpha,n,m,d)$ such that for any $t_2\in (t_0,T_w)$:
    \begin{equation}
        \label{eq-case-t0-leq-t*-done}
         \norm{u_0}_{L^1}^{-\frac{n(\alpha+n+1)}{\alpha+1}-\frac{(\gamma+d(\alpha+n)+4)n^2}{(\alpha+1)(dn+4)}} \, t_0^{\frac{(\alpha d-\gamma)n}{(dn+4)(\alpha+1)}} \,  \gtrsim_{(\alpha,\gamma, n,m,d)} \,\norm{u_0}_{L^1}^{-\frac{n^2}{(dn+4)(\alpha+1)}\left(\gamma+4\left( \frac{\alpha+1}{n} \right)+d\right)} t^{\frac{(\alpha d-\gamma)n}{(dn+4)(\alpha+1)}} \, \bigg\vert_{t_0}^{t_2}.
    \end{equation}
    We note that 
    \[
     -\frac{n(\alpha+n+1)}{\alpha+1}-\frac{(\gamma+d(\alpha+n)+4)n^2}{(\alpha+1)(dn+4)} =-\frac{n^2}{(dn+4)(\alpha+1)}\left(\gamma+4\left( \frac{\alpha+1}{n} \right)+d\right).
    \]
    Hence, just as before, the $\norm{u_0}_{L^1}$-terms in \eqref{eq-case-t0-leq-t*-done} cancel, leaving us with 
    \begin{equation} \label{eq-n-dominates}
        t_2 \, \lesssim_{(\alpha,\gamma, n,m,d)} \, t_0.
    \end{equation}
    Since this holds for all $t_2\in (t_0,T_w)$, it also holds for $T_w$.

    We arrive now at the case $t^*\in (t_0,T_w)$. The ambiguity in this case arises from the fact that the scaling of $R(t)$ changes between $t_0$ and $T_w$, which means that that neither \eqref{eq-t0-geq-t*-R(t)} nor \eqref{eq-t0-leq-t*-R(t)} hold on the entire interval $[t_0,T_w)$. Nonetheless, \eqref{eq-t0-leq-t*-R(t)} holds for all $t\in [t_0,t^*]$. On this interval, then, the previous estimates may be applied (with $t_1 = t_0$ and $t_2=t^*$) to obtain 
    \[
        t^* \, \lesssim_{(\alpha,\gamma,n,m,d)} \, t_0.
    \]
    Furthermore, we see that for $t\in [t^*,T_w)$, the estimate \eqref{eq-t0-geq-t*-R(t)} holds (using $t^*>t_0$). Thus, the previous estimates can also be used to derive \eqref{eq-m-dominates}, but with $t_1=t^*$, i.e.,
    \[
        T_w\, \lesssim_{(\alpha,\gamma,n,m,d)} \, t^*.
    \]
    Chaining the last two inequalities together gives the desired estimate. To complete the proof, we need simply take $c_4>0$ to be the largest of the three implicit constants found in the different cases we have considered. 

\end{proof}

\bigskip
\noindent
\textbf{Acknowledgements.}
J.U.\ has been supported by the Graduiertenkolleg 2339 IntComSin
”Interfaces, Complex Structures, and Singular Limits” of the Deutsche Forschungsgemeinschaft (DFG, German Research Foundation) with Project-ID 321821685. The support
is gratefully acknowledged. J.U. also thanks G. Gr\"un for many useful discussions about the topic. Thanks are also given to R.M. Taranets  for raising attention to some points in the literature that were missed in the original version. 

\vspace{0.1cm}
\noindent
\textbf{AI Usage Declaration.}
At no point in the creation of this manuscript, or in the research process for the results contained herein, has generative AI of any sort been used.

\bibliographystyle{amsplain}
\nocite{*}
\bibliography{bibliography}
%\addcontentsline{toc}{section}{bibliography}

\end{document}